\documentclass[brochure,12pt]{bourbaki}
\usepackage[matrix,arrow]{xy}
\usepackage{amssymb,amsfonts,amsmath,footnote}
\usepackage{color}
\usepackage[francais]{babel}
\usepackage{hyperref}
\newcommand\R{{\mathbb R}}

\newcommand{\cA}{\mathcal A}

\newcommand{\cC}{\mathcal C}

\newcommand{\cE}{\mathcal E}
\newcommand{\cF}{\mathcal F}

\newcommand{\cH}{\mathcal H}

\newcommand{\cJ}{\mathcal J}

\newcommand{\cL}{\mathcal L}
\newcommand{\cM}{\mathcal M}

\newcommand{\cP}{\mathcal P}

\newcommand{\cS}{\mathcal S}

\newcommand{\cU}{\mathcal U}
\newcommand{\cV}{\mathcal V}

\newcommand{\cX}{\mathcal X}

\newcommand{\cZ}{\mathcal Z}

\newcommand\dt{{\frac{\mathrm d}{\mathrm dt}}}
\newcommand\dpt{{\frac{\partial}{\partial t}}}

\newcommand\Dx{{\nabla_x}}

\newcommand\Dv{{\nabla_v}}
\newcommand{\dd}{{\, \mathrm d}}
\newcommand\ddxv{{ \, \mathrm dx \, \mathrm dv}}

\newcommand{\td}[2]{\frac{\mathrm d #1}{\mathrm d #2}}

\newcommand{\rint}{{\int_{\R^3}}}
\newcommand{\drint}{{\int_{\R^6}}}

\newcommand{\vs}{\vspace{0.3cm}}
\newcommand{\ds}{\displaystyle}
\newcommand{\fa}{\forall \,}
\newcommand{\sk}{\smallskip}

\let\oldmarginpar\marginpar
\renewcommand\marginpar[1]{\-\oldmarginpar[\raggedleft\footnotesize #1]%
{\raggedright\footnotesize #1}}

\date{Novembre 2011} \bbkannee{64\`eme ann\'ee, 2011-2012}
\bbknumero{1044} \title{Stabilit\'e orbitale\\ pour le syst\`eme de
  Vlasov-Poisson gravitationnel} \subtitle{d'apr\`es
  Lemou-M\'ehats-Rapha\"el, Guo, Lin, Rein et al.}  \author{Cl\'ement
  MOUHOT}
\address{University of Cambridge\\
  Centre for Mathematical Sciences\\
  Wilberforce Road\\
  Cambridge CB3 0WA, UK} \email{C.Mouhot@dpmms.cam.ac.uk}

\begin{document}
\maketitle

\noindent{\bf INTRODUCTION}

\bigskip

Cet expos\'e est consacr\'e aux avanc\'ees math\'ematiques r\'ecentes
sur un probl\`eme c\'el\`ebre de l'astrophysique, la stabilit\'e de
mod\`eles galactiques. La question se formule tr\`es simplement~: si
l'on consid\`ere un ensemble d'un tr\`es grand nombre d'\'etoiles en
interaction gravitationnelle\footnote{Le r\^ole jou\'e par les plan\`etes 
  dans la dynamique est n\'eglig\'e au premier ordre car leurs masses sont bien plus
  petites que celles des \'etoiles.} dont la coh\'erence est assur\'ee
par leur attraction r\'eciproque et que l'on consid\`ere en premi\`ere
approximation comme un syst\`eme ferm\'e, quelles sont les
r\'epartitions statistiques stables au cours du temps~?
Autrement dit, quelles sont les configurations de galaxies observables
dans notre univers? On n\'eglige ici les effets relativistes et on se
place dans le cadre de la m\'ecanique classique.

Ce probl\`eme semble \`a premi\`ere vue tout autant insoluble que le
probl\`eme \`a $N$ corps de Newton. Comment formuler des pr\'edictions
\`a long terme sur un syst\`eme de $10^{11}$ corps\footnote{C'est
  l'ordre de grandeur du nombre d'\'etoiles dans notre galaxie.}
alors m\^eme que l'on ne sait pas r\'esoudre de mani\`ere
satisfaisante le probl\`eme \`a $3$ corps~! C'est ici que la
m\'ecanique statistique entre en jeu. En suivant les id\'ees de
Maxwell et de Boltzmann, on peut tenter de d\'ecrire de mani\`ere
statistique l'\'evolution de nos $N$~corps lorsque $N$ tend vers
l'infini et que les corps sont suffisamment \og peu corr\'el\'es\fg{},
\`a travers une \'equation aux d\'eriv\'ees partielles non-lin\'eaire
sur la r\'epartition d'un corps typique.

Cette approche a d'abord \'et\'e appliqu\'ee aux cas de gaz
collisionnels pour donner la c\'el\`ebre \'equation de
Boltzmann. Cependant, les collisions entre \'etoiles dans une galaxies
sont quasi-absentes et l'interaction se fait essentiellement \`a
distance \textit{via} le champ gravitationnel. Dans les ann\'ees 1910
et 1930, Jeans puis Vlasov d\'ecouvrent comment effectuer la limite $N
\to +\infty$ dans ce cas \og non-collisionnel\fg{} afin d'obtenir des
\'equations aux d\'eriv\'ees partielles non-lin\'eaires dites de \og
champ moyen\fg{}. La plus c\'el\`ebre d'entre elles est l'\'equation
de Vlasov-Poisson, dont nous allons \'etudier ici la version
gravitationnelle \eqref{eq:1}-\eqref{eq:2}. Cette \'equation de
transport non-lin\'eaire d\'ecrit avec une excellente pr\'ecision
l'\'evolution de syst\`emes stellaires sur de grandes \'echelles de
temps.

Ainsi, lorsque le nombre de corps est grand et que les corr\'elations
sont faibles, on peut esp\'erer formuler des pr\'edictions de
stabilit\'e \`a partir de l'\'etude de l'\'equation de
Vlasov-Poisson. C'est le physicien russe Antonov \cite{An1,An2} qui,
le premier, d\'ecouvre comment r\'esoudre le probl\`eme de la
stabilit\'e, mais uniquement dans un cadre lin\'earis\'e. Il
d\'emontre ainsi la \emph{stabilit\'e lin\'earis\'ee des mod\`eles
  galactiques sph\'eriques et monotones en l'\'energie microscopique
  pour des petites perturbations.}  Cependant, l'\'equation de
Vlasov-Poisson est r\'eversible en temps, non dissipative, et ses
solutions montrent des oscillations en temps grand; rien ne garantit
\textit{a priori} que l'\'etude de stabilit\'e lin\'earis\'ee
d'Antonov implique la stabilit\'e non-lin\'eaire recherch\'ee. C'est
ce probl\`eme math\'ematique que nous appellerons {\bf conjecture de
  stabilit\'e non-lin\'eaire} \`a la Antonov. Durant les cinquante
derni\`eres ann\'ees, l'analyse math\'ematique des \emph{\'equations
  cin\'etiques} de Boltzmann et Vlasov a connu un fort
d\'eveloppement. Cette conjecture de stabilit\'e non-lin\'eaire a
\'et\'e r\'ecemment d\'emontr\'ee par Lemou, M\'ehats et Rapha\"el
\cite{LMR6,LMR8,LMR9}, suivant des travaux pr\'ecurseurs de Dolbeault,
Guo, Had\v zi\'c, Lin, Rein, S\'anchez, Soler, Wan, Wolansky ainsi que
Lemou, M\'ehats et Rapha\"el
\cite{Wa1,Wo-i,Wa2,Gu1,Gu2,Gu3,Gu4,GR1,GR2,DSS,SS,Ha,GR3,GL,LMR1,LMR3,LMR4}.

Dans cet expos\'e, nous proposerons tout d'abord dans la section 1 une
introduction math\'ematique au probl\`eme, en expliquant l'origine du
syst\`eme de Vlasov-Poisson gravitationnelle \`a partir du probl\`eme
\`a $N$ corps. Nous rappellerons ensuite dans la section~2 les
principales propri\'et\'es math\'ematiques de ce syst\`eme
d'\'equations. Puis nous aborderons dans la section~3 la r\'esolution
du probl\`eme lin\'earis\'e. Enfin nous traiterons dans la section~4
de la question de la stabilit\'e non-lin\'eaire. Nous retracerons
l'histoire des diff\'erentes m\'ethodes math\'ematiques
d\'evelopp\'ees pour attaquer ce probl\`eme, puis nous d\'etaillerons
le th\'eor\`eme central du travail \cite{LMR9} de Lemou, M\'ehats et
Rapha\"el, en donnant un sch\'ema d\'etaill\'e de preuve.  Nous
conclurons finalement avec des commentaires et questions ouvertes.

L'auteur remercie Y. Guo, P.-E. Jabin, M. Lemou, F. M\'ehats, Z. Lin,
P. Rapha\"el et J. Soler pour les \'echanges par courriel ou
discussions durant la pr\'eparation de cet expos\'e, ainsi que
S. Martin pour ses relectures et commentaires sur ce texte.

\section{\textup{Probl\`eme \`a $N$ corps et syst\`eme de
    Vlasov-Poisson}}

\subsection{Le probl\`eme \`a $N$ corps de Newton}
\label{sec:le-probleme-a}

On \'ecrit le probl\`eme \`a $N$ corps, pour des interactions binaires
et sans champ ext\'erieur, et en normalisant toutes les masses \`a
$1$~:
\begin{equation}
  \label{eq:4}
  1 \le i \le N, \quad \left\{ 
    \begin{array}{l} \displaystyle
      x_i'(t) = v_i(t) = \frac{\partial H}{\partial v_i}(X(t),V(t)), 
      \vspace{0.3cm} \\ \displaystyle
      v_i'(t) = - \sum_{i \not=  j} \nabla_x \psi(x_i - x_j) 
      = - \frac{\partial H}{\partial x_i}(X(t),V(t))
    \end{array}
    \right.
\end{equation}
o\`u $\psi$ est le potentiel de l'interaction. Ces \'equations sont
\'ecrites sous forme hamiltonienne pour le hamiltonien
\begin{equation}\label{eq:HamMicro}
H(X,V) = \sum_{i=1} ^N \frac{|v_i|^2}2 + \sum_{i < j} \psi(x_i-x_j) 
\end{equation}
o\`u $X=(x_1, \dots, x_N)$, $V=(v_1,\dots, v_d)$ et chaque $x_i, v_i
\in \R^3$. 

M\^eme si ce syst\`eme d'\'equations diff\'erentielles ordinaires
non-lin\'eaires coupl\'ees peut \^etre aussi bien utilis\'e pour
d\'ecrire l'\'evolution d'objets infiniment petits, comme des
mol\'ecules ou infiniment grands, comme des \'etoiles, nous parlerons
ici de \emph{point de vue microscopique et trajectoriel} et nous
appellerons $\psi$ le \emph{potentiel d'interaction microscopique} et
$H$ le \emph{hamiltonien microscopique}.

\subsection{La limite de Vlasov ou limite \og de champ moyen\fg{}}
\label{sec:la-limite-de}

Si l'on consid\`ere l'\'evolution de la distribution $f^{(N)}=f^{(N)}(t,X,V)$
de nos $N$ corps dans l'espace des phases, des positions et vitesses,
on obtient l'\emph{\'equation de Liouville \`a $N$ corps}, qui est le
point de vue statistique sur ce syst\`eme:
\begin{equation} \label{lec2-FNevo} 
  \partial_t f^{(N)} +
  \sum_{i=1}^N \left( \frac{\partial H}{\partial v_i} \cdot
    \frac{\partial f^{(N)}}{\partial x_i} - \frac{\partial H}{\partial x_i}
    \cdot \frac{\partial f^{(N)}}{\partial v_i} \right) = \partial_t f^{(N)} +
  \left\{ f^{(N)}, H
  \right\} = 0
\end{equation} 
o\`u $\{ \cdot, \cdot \}$ d\'esigne le crochet de Poisson sur
$(\R^3)^N \times (\R^3)^N$: 
\[
\{ f^{(N)}, g^{(N)} \} = \nabla_X f^{(N)} \cdot \nabla_V g^{(N)} -
\nabla_X g^{(N)} \cdot \nabla_V f^{(N)}.
\]
Les solutions sont donn\'ees par $f^{(N)}(t,X,V) =
f^{(N)}(S_{-t}(X,V))$ o\`u $S_{t}$ d\'esigne le flot solution des
\'equations de Newton.

Cette r\'esolution par la \emph{m\'ethode des caract\'eristiques}
implique la conservation du signe de $f^{(N)}$:
\[
f^{(N)}(0,X,V) \ge 0 \quad \Longrightarrow \quad f^{(N)}(t,X,V)  \ge 0
\]
mais aussi, du fait que le jacobien de $S_t$ vaut $1$ pour tout $t\,$
\footnote{Ce qui est une version du c\'el\`ebre th\'eor\`eme de
  Liouville dans ce contexte.}, de la masse
\[
\int_{(\R^3\times\R^3)^N} f^{(N)}(t,X,V) \dd X \dd V = 
\int_{(\R^3\times \R^3)^N} f^{(N)}(0,X,V) \dd X\dd V
\]
et plus g\'en\'eralement de toute \og fonctionnelle de Casimir\fg{} 
\[
\int_{(\R^3\times\R^3)^N} \cC \left( f^{(N)}(t,X,V) \right) \dd X \dd V = 
\int_{(\R^3\times \R^3)^N} \cC \left( f^{(N)}(0,X,V) \right) \dd X \dd V
\]
pour $\cC \in C^1(\R_+,\R_+)$, $\cC(0)=0$. Une fonctionnelle de
Casimir d\'esigne pour un syst\`eme hamiltonien une fonctionnelle qui
annule le crochet de Poisson avec toute autre fonctionnelle le long
des trajectoires (voir par exemple \cite{Ar1,Ar2}).

Afin de comprendre l'\'evolution \`a travers un syst\`eme r\'eduit, on
introduit la distribution d'une particule $f = f^{(1)}$:
\[ 
f(t,x,v) := \int_{(\R^3\times \R^3)^{N-1}} f^{(N)}(t,X,V) \dd x_2 \dd
x_3 \, \dots \dd x_N \dd v_2 \dd v_3 \, \dots \dd v_N
\] 
qui est donn\'ee par la marginale de la distribution compl\`ete \`a
$N$ corps $f^{(N)}$. On calcule alors l'\'equation v\'erifi\'ee par cette
distribution \`a une particule 
\begin{equation}\label{eq:oneparticle}
  \frac{\partial f}{\partial t} + v \cdot \nabla_x f 
  - (N-1) \int_{\R^3 \times \R^3} \nabla_ x \psi(x-y)
  \cdot \nabla_v  f^{(2)} (x,y,v,v_*) \dd y \dd v_* = 0.
\end{equation}

Cette \'equation d\'epend bien s\^ur de la seconde marginale
(distribution \`a deux particules) 
\[
f^{(2)}(t,x_1,x_2,v_1,v_2) := \int_{(\R^3\times\R^3)^{N-2}} f^{(N)}
\dd x_3 \dd x_4 \, \dots \dd x_N \dd v_3 \dd v_4 \, \dots \dd v_N
\]
qui contient de l'information sur les corr\'elations du syst\`eme. De
la m\^eme fa\c con, on pourrait d\'efinir les $k$-marginales $f^{(k)}$
pour $1 \le k \le N$, et l'ensemble des \'equations coupl\'ees sur
toutes ces marginales constitue la \emph{hi\'erarchie BBGKY}
(Bogoliubov-Born-Green-Kirkwood-Yvon).

La limite de champ moyen consiste: 
\begin{itemize}
\item d'une part \`a consid\'erer que chaque interaction binaire est
  d'ordre $O(1/N)$, soit
\[
\psi = \psi_N = \frac{\bar \psi}N,
\]
ce qui signifie que l'influence d'une \'etoile sur une autre \'etoile
est n\'egligeable \`a l'\'echelle de l'ensemble de la galaxie, mais
que l'action de l'ensemble de la galaxie sur une \'etoile donn\'ee est
d'ordre $O((N-1)/N)=O(1)$;
\item d'autre part \`a consid\'erer que
\begin{equation}
  \label{eq:chaos}
  f^{(2)} \sim f \times f \quad \mbox{ lorsque } \ N \to \infty,
\end{equation}
hypoth\`ese que l'on appelle \emph{propri\'et\'e de chaos} dans ce
contexte, et qui traduit des \emph{corr\'elations faibles} dans le
syst\`eme.
\end{itemize}

Si l'hypoth\`ese \eqref{eq:chaos} reste vraie au cours du temps
lorsqu'elle est v\'erifi\'ee au temps initial, on parle de \og
\emph{propagation du chaos}\fg{}. D\'emontrer cette propri\'et\'e est
un probl\`eme math\'ematique ouvert dans le cas de l'interaction
gravitationnelle (mais \'egalement pour l'interaction de Coulomb
r\'epulsive entre particules de m\^eme charge); nous renvoyons \`a
\cite{Do,BH,HJ1,HJ2} pour des r\'esultats partiels sur ce sujet.

Sous r\'eserve de l'hypoth\`ese d'\'echelle sur $\psi_N$ et de la
propagation du chaos, on obtient l'\'equation \emph{ferm\'ee} suivante
sur la distribution $f$ dans la limite $N \to +\infty$:
\begin{equation}\label{eq:oneparticleclosed}
  \frac{\partial f}{\partial t} + v \cdot \nabla_x f 
  - \int_{\R^3 \times \R^3} \nabla_ x \bar \psi(x-y)
  \cdot \nabla_v  f (x,v) \, f(y,w) \dd y \dd w = 0,
\end{equation}
que l'on peut r\'e\'ecrire en 
\begin{equation}\label{eq:oneparticleclosed}
  \frac{\partial f}{\partial t} + v \cdot \nabla_x f 
  - \left( \int_{\R^3} \nabla_ x \bar \psi(x-y) \, \left(
      \int_{\R^3} f(y,v_*) \dd v_* \right) \dd y \right) \cdot 
  \nabla_v  f (x,v) = 0.
\end{equation}

\subsection{Le syst\`eme de Vlasov-Poisson gravitationnel}
\label{sec:le-systeme-de}

Lorsque le potentiel microscopique correspond aux interactions
gravitationnelles $\bar \psi(x) = - 1/(4 \pi |x|)$, on obtient ainsi
le \emph{syst\`eme de Vlasov-Poisson gravitationnel}: 
\begin{equation}
  \label{eq:1}
  \left\{
  \begin{array}{l}
    \partial_t f + v \cdot \nabla_x f - \nabla_x \phi_f \cdot \nabla_v
    f =0, \quad t \in \R_+, \ x \in \R^3, \ v \in \R^3,\vspace{0.3cm} \\
    f(t=0,x,v) = f_{in}(x,v)
  \end{array}
\right.
\end{equation}
o\`u l'on d\'efinit le \emph{champ moyen gravitationnel} $\phi_f$
(d\'ependant de la fonction $f$) par
\begin{equation}
  \label{eq:2}
  \rho_f (t,x) := \int_{\R^3} f(t,x,v) \dd v \quad \mbox { et } \quad
  \phi_f(t,x) := - \frac{1}{4 \pi |x|} \ast \rho_f.
\end{equation}
Observons que le champ $\phi_f$ se r\'esout par l'\emph{\'equation
elliptique de Poisson}
\begin{equation}
  \label{eq:3}
  \Delta_x \phi_f = \rho_f
\end{equation}
ce qui conf\`ere son nom au syst\`eme de Vlasov-Poisson.

On voit ici le caract\`ere fondamental du syst\`eme de
Vlasov-Poisson. Dans le cas d'inter\-actions \`a distance (par
opposition aux m\'ecanismes de collision), il est \`a la m\'ecanique
classique statistique ce que les \'equations de Newton sont \`a la
m\'ecanique classique.

\subsection{Le probl\`eme de stabilit\'e d'un point de vue statistique}
\label{sec:le-probleme-de}

On peut maintenant reformuler la question de la stabilit\'e des
galaxies dans le cadre du syst\`eme de Vlasov-Poisson. Cela rejoint
l'une des premi\`eres questions que l'on se pose face \`a une
\'equation aux d\'eriv\'ees partielles: quelles sont les solutions
stationnaires de ce syst\`eme et, parmi ces derni\`eres, lesquelles
sont stables ou instables?

De nouveaux aspects math\'ematiques diff\'erents de l'analyse des
\'equations de Newton, et sur lesquels nous reviendrons plus loin,
surgissent alors: quelles notions de solution utiliser pour
l'\'equation? Quels espaces fonctionnels utiliser?  Quelle
r\'egularit\'e pour ces solutions?

Puisque nous sommes partis des \'equations de Newton, il est
l\'egitime de se demander quel serait l'impact en retour d'un
th\'eor\`eme de stabilit\'e sur le probl\`eme \`a $N$~corps. C'est une
question difficile et mal comprise \`a l'heure actuelle. Nous pouvons
n\'eanmoins avancer deux principes~: d'une part, cela d\'epend
fortement des r\'eponses donn\'ees aux probl\`emes ouverts de limite
de champ moyen et propagation du chaos et, d'autre part, cela d\'epend
vraisemblablement \'egalement de l'espace fonctionnel dans lequel les
r\'esultats de stabilit\'e sont obtenus. Dans cet expos\'e, nous nous
int\'eresserons \`a des voisinages de stabilit\'e pour des espaces des
Lebesgue; ces derniers sont peut-\^etre trop peu r\'eguliers pour
esp\'erer en d\'eduire des informations sur le probl\`eme \`a
$N$~corps.

\section{\textup{Propri\'et\'es du syst\`eme de Vlasov-Poisson}}
\label{sec:les-propr-fond}

Nous allons tout d'abord rappeler les propri\'et\'es math\'ematiques
\'el\'ementaires des \'equations \eqref{eq:1}-\eqref{eq:2}. 

\subsection{\'Energie microscopique et hamiltonien}
\label{sec:energ-micr-et}

Les \'equations de Newton pr\'eservent le hamiltonien microscopique
$H$ d\'efini par \eqref{eq:HamMicro}:
\[
\forall \, t \in \R, \quad H(S_t(X,V)) = H(X,V). 
\]
En int\'egrant ceci contre la distribution initiale de nos $N$
particules $f^{(N)}$ (\'equation de Liouville) et en divisant par $N$
(pour \'eviter la divergence de l'\'energie lorsque $N$ tend vers
l'infini), on obtient la conservation statistique du hamiltonien
microscopique
\[
\fa t \in \R, \int_{(\R^3 \times \R^3)^N} \!\frac{H(X,V)}{N} \,
f^{(N)}(t,X,V) \dd X \dd V\! =\! \!\int_{(\R^3 \times \R^3)^N}\! \frac{H(X,V)}{N} \,
f^{(N)}(0,X,V) \dd X \dd V.
\] 
Dans la limite de champ moyen, cette \'egalit\'e implique la
conservation au cours du temps de la quantit\'e
\[
\mathcal H(f) = \int_{\R^3 \times \R^3} \frac{|v|^2}{2} \, f(t,x,v) \dd x \dd v+
\int_{\R^3 \times \R^3} \frac{\phi_f (t,x)}{2} \, f(t,x,v) \dd x \dd v
\]
que l'on appellera \emph{hamiltonien (macroscopique) du syst\`eme de
  Vlasov-Poisson}. Remarquons le facteur $1/2$ devant $\phi_f$ qui
rappelle que notre \'equation sur la distribution \`a une particule
$f$ est la \og trace\fg{} d'un syst\`eme \`a $N$ particules avec
interaction binaire. Cette fonctionnelle se r\'e\'ecrit en utilisant
l'\'equation de Poisson \eqref{eq:3} sur $\phi_f$:
\begin{equation}\label{eq:HamMacro}
  \mathcal H(f) = \int_{\R^3 \times \R^3} \frac{|v|^2}{2} \, f(t,x,v) 
  \dd x \dd v - 
  \int_{\R^3 \times \R^3} \frac{|\nabla_x \phi_f (t,x)|^2}{2} 
  \dd x \dd v.
\end{equation}
On voit ici la possibilit\'e de transferts d'\'energie entre \'energie
cin\'etique et gravitationnelle. De plus, le caract\`ere
\emph{attractif} de l'interaction se traduit par le fait que ces
transferts peuvent \og diverger\fg{}: la fonctionnelle $\mathcal H$
conserv\'ee est la somme de deux termes de signe oppos\'e qui peuvent
potentiellement diverger tout en s'\'equilibrant, ce qui est une
source de difficult\'es math\'ematiques.

Notons enfin que l'\'equation de Vlasov de champ moyen \eqref{eq:1}
peut \^etre interpr\'et\'ee comme l'\'equation de Liouville sur une
particule associ\'ee \`a l'\emph{hamiltonien microscopique de champ
  moyen}
\[
E = E_{f}(t,x,v) = \frac{|v|^2}{2} +
\phi_f(t,x)
\]
sous la forme 
\[
\partial_t f + \left\{ f, E_f \right\} = 0
\]
o\`u $\{ \cdot, \cdot \}$ d\'esigne ici le crochet de Poisson sur
$\R^3 \times \R^3$:
\[
\{ f, g \} = \Dx f \cdot \Dv g - \Dx g \cdot \Dv f. 
\]
Remarquons qu'il n'y a \emph{plus} le facteur $1/2$ devant l'\'energie
potentielle dans la d\'efinition de $E_f$: la cause en est que l'on
consid\`ere ici une particule \'evoluant dans un champ moyen donn\'e
en \og oubliant\fg{} la non-lin\'earit\'e.

\subsection{Fonctionnelles de Casimir et \'equimesurabilit\'e}
\label{sec:fonct-de-casim}

L'\'equation de champ moyen \eqref{eq:1} est associ\'ee aux solutions
(dites \emph{caract\'eristiques}) du syst\`eme diff\'erentiel 
\[
\left\{ 
    \begin{array}{l} \displaystyle
      x'(t) = v(t) = \frac{\partial E_{f}}{\partial v}, 
      \vspace{0.3cm} \\ \displaystyle
      v'(t) = - \nabla_x \phi_f(x(t))
      = - \frac{\partial E_f}{\partial x}.
    \end{array}
    \right.
\]
Ces courbes d\'ependent bien s\^ur de la solution elle-m\^eme, mais
cela montre que l'\'equation d'\'evolution \eqref{eq:1} est une
\emph{dynamique de r\'earrangement} de la distribution $f$, qui
pr\'eserve le signe, mais aussi toutes les fonctionnelles de Casimir 
\[
\int_{\R^3\times\R^3} \cC \left( f(t,x,v) \right) \dd x \dd v = 
\int_{\R^3\times \R^3} \cC \left( f(0,x,v) \right) \dd x \dd v
\]
pour $\cC \in C^1(\R_+,\R_+)$, $\cC(0)=0$. Cet ensemble de lois de
conservation peut \^etre synth\'etiquement exprim\'e par la
propri\'et\'e suivante: \emph{la solution reste pour tout temps
  \'equimesurable \`a sa donn\'ee initiale}
\[
\forall \, t \ge 0, \ \forall \, s \ge 0, \quad \left| \left\{
    |f(t,x,v)| \ge s \right\} \right| =\left| \left\{ |f_{in}(x,v)|
    \ge s \right\} \right| \in \R_+ \cup \{ + \infty\}
\]
o\`u $|\mathcal O|$ d\'esigne la mesure de Lebesgue d'un ensemble
mesurable $\mathcal O \subset \R^3 \times \R^3$.

Nous sommes donc en pr\'esence d'un syst\`eme dynamique en dimension
infinie \emph{poss\'edant une infinit\'e non d\'enombrable de
  contraintes}: le hamiltonien macroscopique $\cH(f)$ et
l'\'equimesurabilit\'e.

\subsection{Lien entre invariants et hamiltonien: une
  in\'egalit\'e d'interpolation clef}
\label{sec:lien-entre-fonct}

Comme nous l'avons vu, l'\'energie totale du syst\`eme $\cH(f)$ est la
diff\'erence de deux termes positifs que nous noterons
\[
\cH(f) = \cH_{cin}(f) - \cH_{pot}(f).
\]
Une difficult\'e \'evidente sera alors de contr\^oler ces deux formes
d'\'energie s\'epar\'ement. \`A cette fin, il est possible
\emph{d'interpoler} l'\'energie potentielle \`a partir de l'\'energie
cin\'etique et d'un espace de Lebesgue~$L^p$, dont la conservation est
assur\'ee par l'invariance des fonctionnelles de Casimir. On peut
ainsi montrer les in\'egalit\'es fonctionnelles suivantes.

\begin{prop}\label{prop:interpol-clef}
  Pour tout $p \in [1,+\infty]$, on a l'estimation d'int\'egrabilit\'e
  de la densit\'e spatiale 
\[
\| \rho_f \|_{L^{\frac{5p-3}{3p-1}} (\R^3)} \le C \| f \|_{L^p(\R^6)} ^{\frac{2p}{5p-3}}
\cH_{cin} (f) ^{\frac{3p-3}{5p-3}}
\]
avec dans le cas limite $p=+\infty$
\[
\| \rho_f \|_{L^{\frac{5}{3}}(\R^3)} \le C \| f \|_{L^\infty(\R^6)} ^{\frac{2}{5}}
\cH_{cin} (f) ^{\frac{3}{5}},
\]
pour une certaine constante $C>0$. 

On en d\'eduit le contr\^ole suivant de l'\'energie potentielle pour
$p >p_c = 9/7$
\[
\cH_{pot}(f) \le C \| f \|_{L^1(\R^6)} ^{\frac{7p-9}{6(p-1)}} \| f \|_{L^p(\R^6)}
^{\frac{p}{3(p-1)}} \cH_{cin}(f)^{\frac{1}{2}}
\]
pour une certaine constante $C>0$. 
\end{prop}

\begin{proof}
  Pour $p>1$ et $R > 0$, on d\'ecompose l'int\'egrale de la densit\'e
  spatiale $\rho_f$ en
\begin{align*}
  \int_{\R^3} f \dd v & = \int_{|v| \leq R} f \dd v + \int_{|v| > R}
  f \dd v \\
  & \leq 
    \| f \|_{L^p _v} 
  \left( \int_{|v|\leq R} \dd v \right)^{\frac{(p-1)}{p}} + R^{- 2}
  \int_{\R^3} f |v|^2 \dd v\\
  & \leq 
  \| f \|_{L^p _v} \left( \frac{4\pi R^{3}}{3} \right)^{\frac{(p-1)}{p}}
  + R^{-2} \rint f |v|^2 \dd v.
\end{align*}
En optimisant le param\`etre $R$, on obtient 
\[\int_{\R^3} f \dd v \leq C 
\Vert f(x,\cdot) \Vert_{L^p (\R^3)}^{\frac{2p}{(5p-3)}} \left( \rint f |v|^2
  \dd v \right)^{\frac{(3p-3)}{(5p-3)}} \] 
et on en d\'eduit par in\'egalit\'e
de Cauchy-Schwarz
\begin{eqnarray*}
  \left\Vert \rho_f \right \Vert_{L^{\frac{(5p-3)}{(3p-1)}}(\R^3)} &\leq& C
  \!\left( \rint \Vert f(x,\cdot)\Vert_{L^p(\R^3)} ^{\frac{2p}{(3p-1)}}\! \left(
      \rint f |v|^2 \dd v
   \! \right)^{\frac{(3p-3)}{(3p-1)}} \dd x\! \right)^{\frac{(3p-1)}{(5p-3)}} \\
  &\leq& C \Vert f \Vert_{L^p(\R^6)}^{\frac{2p}{(5p-3)}} \cH_{cin}(f)^{\frac{(3p-3)}{(5p-3)}},
\end{eqnarray*}
ce qui conclut la preuve de la premi\`ere in\'egalit\'e. 

On veut maintenant contr\^oler le champ $\phi_f$. En raisonnant
heuristiquement, l'in\'egalit\'e de Sobolev et la r\'egularit\'e
elliptique de l'\'equation de Poisson impliqueraient 
\[
\| \Dx \phi_f \|_{L^2(\R^3)} \le C \| \Dx \phi_f \|_{W^{1,6/5}(\R^3)}
\le C \| \rho_f \|_{L^{6/5}(\R^3)}.
\]
La fonction $\nabla_x \phi_f$ n'est pas dans l'espace
$W^{1,6/5}(\R^3)$; cependant on a bien l'in\'egalit\'e $\| \Dx \phi_f
\|_{L^2(\R^3)} \le C \| \rho_f \|_{L^{6/5}(\R^3)}$ en utilisant la
formule $\nabla_x \phi_f = (x/(4 \pi |x|^3)) \ast \rho_f$ et
l'in\'egalit\'e de Hardy-Littlewood-Sobolev \cite[Theorem~4.3]{LL}. De
mani\`ere alternative, il est possible de montrer que $\nabla_x \phi_f
\in \dot W^{1,6/5}$ (l'espace de Sobolev homog\`ene) au moyen du
contr\^ole sur le laplacien de $\phi_f$ et de la th\'eorie de
Calder\'on-Zygmund, puis d'utiliser l'in\'egalit\'e de
Gagliardo-Nirenberg-Sobolev pour conclure.

En combinant ces deux pr\'ec\'edentes in\'egalit\'es, on obtient 
\[
\cH_{pot}(f) = \| \nabla_x \phi_f \|_{L^2(\R^3)} ^2 \le C \| \rho_f
  \|_{L^{6/5}(\R^3)} ^2.
\]
Afin d'estimer ce membre de droite on effectue l'interpolation
suivante: pour \mbox{$p>p_c=9/7$,} on a $(5p-3)/(3p-1) > 6/5$ et
\begin{eqnarray*}
  \| \rho_f \|_{L^{6/5}(\R^3)} &\le& \| \rho_f \|_{L^1 (\R^3)} ^{\frac{(7p-9)}{(12p-12)}} \|
  \rho_f \|_{L^{\frac{(5p-3)}{(3p-1)}}(\R^3)} ^{\frac{(5p-3)}{(12p-12)}} 
  \\ &\le& \| f \|_{L^1 (\R^6)} ^{\frac{(7p-9)}{(12p-12)}} \|
  \rho_f \|_{L^{\frac{(5p-3)}{(3p-1)}}(\R^6)} ^{\frac{(5p-3)}{(12p-12)}}.
\end{eqnarray*}

On combine alors toutes les in\'egalit\'es pr\'ec\'edentes pour
obtenir (toujours lorsque $p>p_c$)
\[
\cH_{pot}(f) \le C \| f \|_{L^1 (\R^6)} ^{\frac{7p-9}{6(p-1)}} \| f \|_{L^p (\R^6)}
^{\frac{p}{3(p-1)}} \cH_{cin}(f)^{\frac{1}{2}}
\]
ce qui termine la d\'emonstration de la seconde in\'egalit\'e.
\end{proof}

Remarquons que l'\'energie potentielle peut \^etre ainsi contr\^ol\'ee
par interpolation entre les invariants du syst\`eme et l'\'energie
cin\'etique, \emph{avec une puissance strictement inf\'erieure \`a $1$
  de l'\'energie cin\'etique} (plus pr\'ecis\'ement $1/2$ dans notre
cas). Par analogie avec les \'equations aux d\'eriv\'ees partielles
dispersives et en particulier l'\'equation de Schr\"odinger, dans le
cas d'une telle \'equation o\`u le hamiltonien est compos\'e de deux
termes de signe oppos\'e qui peuvent diverger tout en s'\'equilibrant,
on parle d'\emph{\'equation focalisante}; et lorsque l'on peut
\'etablir un tel contr\^ole de l'\'energie potentielle, on parle
d'\'equation \emph{sous-critique}. Lorsqu'un tel contr\^ole est
possible mais avec une puissance $1$ de l'\'energie potentielle, on
parle d'\'equation \emph{critique}. Lorsqu'un tel contr\^ole n'est
possible qu'avec une puissance strictement plus grande que $1$, on
parle d'\'equations \emph{sur-critiques}, pour lesquelles on s'attend en
g\'en\'eral \`a des ph\'enom\`enes d'explosion en temps fini.

Dans le cas du syst\`eme de Vlasov-Poisson en dimension $3$ qui est
sous-critique, comme le sugg\`ere le r\'esultat que l'on vient
d'\'etablir, la difficult\'e principale pour montrer des
propri\'et\'es de stabilit\'e est le contr\^ole de l'\'energie
cin\'etique.  Nous renvoyons \`a \cite{LMR1,LMR2,LMR4,LMR5,LMR7} pour
une \'etude syst\'ematique de la formation de singularit\'es dans des
cas critiques (dimension $4$ ou dimension $3$ relativiste), ainsi
que pour des commentaires plus pr\'ecis sur le parall\`ele avec les
\'equations de Schr\"odinger. Ces mod\`eles critiques n'ont
toutefois pas d'interpr\'etation claire au niveau physique. Nous
nous concentrerons dans cet expos\'e sur les r\'esultats qui
concernent le syst\`eme \eqref{eq:1}-\eqref{eq:2} en dimension $3$.

\subsection{La th\'eorie de Cauchy}
\label{sec:la-therorie-de}

La th\'eorie de Cauchy est relativement d\'evelopp\'ee pour cette
\'equation non-lin\'eaire. C'est m\^eme l'un des rares syst\`emes
non-lin\'eaires fondamentaux de la physique pour lequel on sait
construire des solutions globales et uniques sans hypoth\`ese
perturbative en dimension $3$. Il est facile de se convaincre des
points suivants:
\begin{itemize} 
\item on souhaite montrer une r\'egularit\'e suffisante sur le champ
  de force moyen pour pouvoir d\'efinir des courbes caract\'eristiques
  pour les \'equations diff\'erentielles ordinaires correspondantes et
  \'etablir des contr\^oles sur ces derni\`eres;
\item cette r\'egularit\'e sur le champ d\'epend du contr\^ole sur des
  \emph{moments en vitesse} $\int f |v|^k \dd x \dd v$ de la solution,
  donc en d\'efinitive du comportement de la distribution pour les
  grandes vitesses.
\end{itemize}

Le contr\^ole de ces \og particules \`a grandes vitesses\fg{} a
finalement \'et\'e obtenu de deux mani\`eres diff\'erentes il y a une
vingtaine d'ann\'ees. D'une part, Pfaffelmoser \cite{Pf} a d\'ecouvert
comment propager en temps une borne sur la taille du support de la
solution, ce qui implique le contr\^ole recherch\'e et permet de
construire des solutions globales $C^1$ \`a support compact. D'autre
part et ind\'ependamment, Lions et Perthame \cite{LP} ont d\'ecouvert
comment propager en temps des bornes directement sur les moments en
vitesse de toute solution dans $L^\infty$ et de hamiltonien born\'e,
ce qui leur permet de construire des solutions uniques et globales
dans $L^1((1+|v|^k) \dd x \dd v) \cap L^\infty$ pour $k$ suffisamment
\'elev\'e, modulo une hypoth\`ese assez faible de r\'egularit\'e
initiale sur la densit\'e spatiale $\rho_{in}$. Nous renvoyons ensuite
aux articles \cite{Sc1,BR,Ho-ii,Wo-ii,Ro,GJP,Lo,MMP,Pa} pour les
d\'eveloppements ult\'erieurs de ces th\'eories.

Par ailleurs, sous des hypoth\`eses plus faibles sur les donn\'ees
initiales
\[
f_{in} \ge 0, \quad f_{in} \in L^1 \cap L^\infty(\R^6), \quad |v|^2 \,
f_{in} \in L^1(\R^6), 
\]
des solutions faibles globales ont \'et\'e construites par Arsen'ev
\cite{Ar} (voir \'egalement \cite{IN,HH}); cependant leur unicit\'e
reste ouverte. Ces solutions sont \'egalement des solutions
renormalis\'ees au sens de DiPerna-Lions \cite{DL1,DL2} et on peut
leur associer des caract\'eristiques g\'en\'eralis\'ees au sens de
DiPerna-Lions \cite{DL3}. Elles sont continues en temps \`a valeur
dans $L^1(\R^6)$ et v\'erifient l'in\'egalit\'e suivante
\[
\fa t \ge 0, \quad \cH(f_t) \le \cH(f_{in})
\]
(la conservation exacte du hamiltonien est ouverte).  Les
th\'eor\`emes de Lemou, M\'ehats et Rapha\"el que nous pr\'esenterons
dans la derni\`ere section ont pour objet ces solutions faibles.

\subsection{L'\'etude des solutions stationnaires}
\label{sec:letude-des-solutions}

Nous cherchons maintenant les solutions stationnaires du syst\`eme
\eqref{eq:1}-\eqref{eq:2}. Il en existe en fait une infinit\'e, dont
nous ne connaissons et ne savons caract\'eriser de mani\`ere
satisfaisante qu'une petite partie. Une telle solution $f^0$ doit
v\'erifier
\[
\left\{ 
\begin{array}{ll}
\forall \, x,v \in \R^3, \quad &v \cdot \nabla_x f^0 - \nabla_x U(x)
\cdot \nabla_v
f^0 =0, \vspace{0.3cm} \\ \displaystyle
\forall \, x \in \R^3, \quad &\Delta_x U = \rho_{f^0}.
\end{array}
\right.
\]

On voit que la premi\`ere \'equation incite \`a consid\'erer une
fonction $f^0$ ne d\'ependant que des invariants du mouvement du
syst\`eme diff\'erentiel $x''(t) = - \nabla_x U$, au moins si ce
dernier est compl\`etement int\'egrable. Cependant, on ne conna\^it
rien \textit{a priori} sur ce syst\`eme diff\'erentiel, puisqu'il
d\'epend de la solution $f^0$ recherch\'ee elle-m\^eme.

\subsubsection{Le th\'eor\`eme de Jeans} 
\label{subsec:le-theoreme-de}

On sait n\'eanmoins r\'esoudre ce probl\`eme dans le cas de solutions
stationnaires dites \`a \og sym\'etrie sph\'erique\fg{}. Cela
correspond \`a l'invariance de la solution stationnaire par les
rotations du r\'ef\'erentiel, soit
\[
\forall \, (x,v) \in \R^6, \quad f^0(Ax,Av) = f^0(x,v)
\]
pour toute matrice $3 \times 3$ orthogonale $A$. Cela implique donc
que $f^0 = f^0(|x|,|v|,x \cdot v)$. Les \'equations du mouvement 
\[
\left\{ 
\begin{array}{l} 
x'(t) = v, \vspace{0.3cm} \\
v'(t) = - \nabla_x U = - \frac{x}{|x|} \, U'(|x|)
\end{array} 
\right.
\]
poss\`edent alors \emph{toujours} (c'est-\`a-dire ind\'ependamment de $f^0$) un
invariant suppl\'e\-mentaire, en plus de l'\'energie~:
\[
\td{}{t} \, \left( x(t) \wedge v(t) \right) = v(t) \wedge v(t) - \frac{U'(|x|)}{|x|}
\, x(t) \wedge x(t) = 0. 
\]
Le syst\`eme est alors compl\`etement int\'egrable (l'espace des
phases poss\`ede trois degr\'es de libert\'e), et l'on peut montrer
que les \'equilibres s'\'ecrivent sous la forme
\begin{equation}\label{eq:jeans}
\left\{
\begin{array}{l}
  f^0(x,v) = f^0(|x|,|v|,x \cdot v) = F\left(E, L\right),
  \vspace{0.3cm} \\ \displaystyle
  E(x,v) := E_{f^0}(x,v) = \frac{|v|^2}{2} +
  \phi_{f^0}(t,x), \vs \\ \ds
  L(x,v) := | x \wedge v|^2.
\end{array}
\right.
\end{equation}
C'est le \emph{th\'eor\`eme de Jeans} \cite{Je1,Je2} appliqu\'e aux
syst\`emes sph\'eriques (voir la discussion \cite[section~4.4]{BT}).
Un cadre et une d\'emonstration math\'ematiques rigoureux peuvent
\^etre consult\'es dans \cite{BFH}, fond\'es sur la r\'esolution des
\'equations diff\'erentielles non-lin\'eaires radiales
correspondantes.

Lorsque $F = F(E)$ ne d\'epend que de $E$, on parle de mod\`eles
sph\'eriques \emph{isotropes}, et lorsque $F=F(E,L)$, on parle de
mod\`eles sph\'eriques \emph{anisotropes}. Ces termes s'expliquent par
le fait que la matrice (tenseur) de dispersion en vitesse
\[
\frac{1}{\rho_f} \rint v_i v_j F(E,L) \dd v, \quad 1 \le i,j, \le 3
\]
est un multiple de la matrice identit\'e lorsque $F=F(E)$ alors que,
par exemple, ses composantes diagonales sont diff\'erentes lorsque
$F=F(E,L)$ d\'epend de $L$ (cf. \`a nouveau, ici et pour la suite de
cette section, \cite[section~4.4]{BT}). 

Dans le cadre de tels syst\`emes \`a sym\'etrie sph\'erique,
l'\'equation fondamentale qui gouverne les \'equilibres de syst\`emes
stellaires est
\begin{equation}\label{eq:stat}
\Delta_x \phi = \rho \quad \Longrightarrow \quad \frac{1}{r^2} \left(
  r^2 \phi' \right)' = \drint F\left( \frac{|v|^2}{2} + \phi, |x \wedge
  v|^2\right) \ddxv.
\end{equation}
C'est une \'equation elliptique que l'on peut voir comme une
\'equation non-lin\'eaire sur le potentiel $\phi$ \'etant donn\'e le
profil de l'\'equilibre sph\'erique $F$, ou bien comme une \'equation
non-lin\'eaire sur $F$ \'etant donn\'e $\phi$.

\subsubsection{Les mod\`eles polytropiques}
\label{sec:les-model-polytr}

Supposons tout d'abord $F=F(E)$ isotrope, et partons de la formule la
plus simple pour obtenir un profil $F$ \`a support compact,
c'est-\`a-dire une fonction puissance tronqu\'ee:
\begin{equation}\label{eq:poly}
f(x,v) = F(E) = C_F (E_0 - E)_+ ^n, \quad n \in \R, 
\end{equation}
o\`u $z_+ := \max \{ z, 0 \}$ d\'esigne la partie positive d'un r\'eel
$z \in \R$, et $C_F$ est une constante donn\'ee. Un calcul
\'el\'ementaire montre que, par r\'esolution de l'\'equation
\eqref{eq:stat}, la densit\'e spatiale associ\'ee est
\[
\rho = c_n (E_0 -\phi)_+ ^{n+3/2}
\]
avec une constante $c_n$ finie si $n >-1$ et la formule simple lorsque
$n \ge 0$: 
\[
c_n := C_F \frac{2^{3/2} \pi n !}{(n+3/2)!}. 
\]
On voit ainsi que ces mod\`eles ne peuvent \emph{jamais \^etre
  r\'epartis de mani\`ere homog\`ene} car une densit\'e spatiale
constante impliquerait $n=-3/2$. L'\'equation obtenue alors sur le
potentiel est l'\'equation dite de \emph{Lane-Emden} 
\[
\frac{1}{r^2} \td{}{r} \left( r^2 \td{}{r} \phi \right) = c_n (E_0
-\phi)_+ ^{n+3/2}.
\]
On parle de mod\`eles \emph{polytropiques} car la distribution spatiale
est alors la m\^eme que pour un gaz auto-gravitant v\'erifiant la loi
de pression des gaz polytropes $p = \mbox{cste} \, \rho^\gamma$ avec
$\gamma = 1+1/(n+3/2)$.  Le cas $n>7/2$ conduit \`a des solutions de
masse infinie et \`a support non compact en ajustant le param\`etre
$E_0$. L'\'equation sur le potentiel n'admet pas de solution simple en
terme de fonctions \'el\'ementaires en dehors du cas $n=-1/2$
(\'equation lin\'eaire de Helmhotz). Le cas limite $n=7/2$ est
int\'eressant car on obtient une masse finie mais un support non
compact; c'est le mod\`ele dit \emph{de Plummer}. 

\subsubsection{Le mod\`ele de King}
\label{sec:le-modele-de}

D'apr\`es les observations, l'\'energie est tr\`es concentr\'ee au
centre des galaxies. Cela incite \`a consid\'erer de grandes valeurs
de $n$ dans le mod\`ele polytropique. Consid\'erons ainsi tout d'abord
formellement le cas limite \og $n=\infty$\fg{}. Bien s\^ur les
\'equations que nous avons utilis\'ees pour trouver la solution
stationnaire ci-dessus ne font plus sens, mais si on suit le
parall\`ele avec les gaz auto-gravitants, cela doit correspondre au
cas d'un gaz v\'erifiant la loi de pression $p = \mbox{cste} \;
\rho^\gamma$ avec $\gamma=1$, c'est-\`a-dire un gaz isotherme. En
suivant les consid\'erations usuelles de m\'ecanique statistique, on
obtient des distributions de loi exponentielle. C'est le mod\`ele dit
de la \emph{sph\`ere isotherme}. Il est cependant de masse infinie et
de support non compact. La litt\'erature physique a donc introduit
\cite{Ki} la troncature suivante
\begin{equation}\label{eq:king}
f = F(E) = \left( e^{E_0-E} - 1\right)_+
\end{equation}
que l'on appelle \emph{mod\`ele de King}. Ce mod\`ele reproduit les
fortes valeurs d'\'energie observ\'ees pour les petits rayons, tout en
\'etant compatible avec les contraintes de masse et support. Ce
mod\`ele sera l'un des protagonistes principaux des travaux
math\'ematiques r\'ecents, car c'est peut-\^etre le mod\`ele
sph\'erique isotrope le plus important d'un point de vue physique, et
l'\'etude de sa stabilit\'e \'echappe aux techniques variationnelles
fond\'ees sur les fonctionnelles d'\'energie-Casimir. Une des
r\'eussites marquantes du programme de recherche de Lemou, M\'ehats et
Rapha\"el est d'obtenir dans \cite{LMR9} la stabilit\'e de ces
mod\`eles pour des perturbations g\'en\'erales.

Remarquons ici que l'on voit appara\^itre en filigrane un
ph\'enom\`eme myst\'erieux, \`a savoir le fait que les \'equilibres
sont analys\'es par des consid\'erations de m\'ecanique statistique
\emph{collisionnelle}, soit irr\'eversible. Comprendre pourquoi cette
d\'emarche produit des r\'esultats corrects, au moins partiellement,
rejoint la question fondamentale de comprendre comment les galaxies se
\og thermalisent\fg{} sans collision. C'est l'objet des th\'eories
d'amortissement et de relaxation violente initi\'ees par Lynden-Bell
\cite{L-B1,L-B2}.

\subsubsection{Autres mod\`eles}
\label{sec:autres-modeles}

Si l'on se prescrit au d\'epart la densit\'e spatiale et que l'on
cherche \`a calculer $F$, on obtient la formule dite de Eddington
\cite{Ed}, dont les deux mod\`eles c\'el\`ebres qui en sont issus sont
le mod\`ele dit $R^{1/4}$ et l'\emph{isochrone d'H\'enon}. Le cas
anisotrope est moins clair, mais l'on peut citer le mod\`ele dit
\emph{d'Osipkov-Merritt} qui g\'en\'eralise le mod\`ele $R^{1/4}$,
ainsi que le mod\`ele dit \emph{de Michie} qui g\'en\'eralise le mod\`ele
de King (nous renvoyons \`a \cite[Chapitre 4]{BT}).

\subsection{La conjecture de stabilit\'e pour les mod\`eles
  sph\'eriques}
\label{sec:la-conjecture-de}

Nous devons d'abord discuter de la notion de stabilit\'e
utilis\'ee. Au vu des invariants du syst\`eme (fonctionnelles de
Casimir et hamiltonien macroscopique), il est naturel de mesurer la
stabilit\'e de la mani\`ere suivante: pour tout $\varepsilon>0$, il
existe $\eta>0$ tel que  
\begin{multline}
  \mathcal D(f_{in}, f^0) := \left\| f_{in} - f^0 \right\|_{L^1(\R^6)}
  + \left| \mathcal H(f_{in}) - \mathcal H(f^0) \right| \le \eta \\
  \quad \Longrightarrow \quad \forall \, t \ge 0, \ \mathcal
  D(f_t,f^0) 
  \le \varepsilon.
\end{multline}
(Les fonctionnelles de Casimir sont correctement captur\'ees par les espaces
de Lebesgue en terme de distance, et le choix de $L^1 (\R^6)$ semble
arbitraire, mais sera de toute fa\c con compl\'et\'e plus loin par une
hypoth\`ese de borne $L^\infty (\R^6)$.) 

Cependant, il faut tenir compte du groupe de sym\'etries du syst\`eme
d'\'evolution et voir si ces derni\`eres sont \og captur\'ees\fg{} par
la distance utilis\'ee pour mesurer la stabilit\'e. Il s'av\`ere que
le syst\`eme de Vlasov-Poisson \eqref{eq:1}-\eqref{eq:2} admet un
large groupe de sym\'etries: si $f(t,x,v)$ est solution, alors pour
tous $t_0 \in \R$, $x_0 \in \R^3$, $\lambda_0 \in \R_+^*$, $\mu_0 \in
\R_+ ^*$
\[
g(t,x,v) := \frac{\mu_0}{\lambda_0^2} \, f\left( \frac{t
    +t_0}{\lambda_0 \, \mu_0}, \frac{x+x_0}{\lambda_0}, \mu_0 \, v
\right)
\]
est encore solution, et on a \'egalement \emph{l'invariance par
  translation galil\'eenne} (ce qui ne fait que traduire sur ce
mod\`ele statistique l'invariance galil\'eenne correspondante de la
m\'ecanique classique): si $f(t,x,v)$ est solution, alors pour
tout $v_0 \in \R^3$
\[
g(t,x,v) := f(t,x+tv_0, v+v_0)
\]
est encore solution. Il est clair que la \og distance\fg{} utilis\'ee
$\mathcal D$ est \'egalement invariante par ces transformations
lorsque $\lambda_0 = \mu_0 = 1$. Nous devons donc reformuler la
stabilit\'e de la mani\`ere suivante: il existe une fonction de
translation $z=z(t) \in \R^3$ telle que
\begin{equation}
  \mathcal D(f_{in}, f^0) \le \eta 
  \quad \Longrightarrow \quad \forall \, t \ge 0, \ \mathcal
  D(f_t,f^0(\cdot -z(t),
  \cdot))
\le \varepsilon.
\end{equation}
C'est la notion de \emph{stabilit\'e orbitale}. Ajoutons que, dans le
cas de perturbations \`a sym\'etrie sph\'erique, le param\`etre $z(t)$
est n\'ecessairement toujours nul.

Pour les solutions stationnaires \`a sym\'etrie sph\'erique, il est
tr\`es largement discut\'e dans la litt\'erature physique \cite{BT}
que la stabilit\'e est li\'ee \`a la monotonie par rapport \`a
l'\'energie microscopique $E=E_{f^0}(x,v)$. Plus pr\'ecis\'ement, on
peut formuler les conjectures suivantes:
\begin{conj} \label{conj1} Les galaxies $f^0= F(E,L)$ \`a sym\'etrie
  sph\'erique anisotropes et d\'ecroissantes ($\partial_E F
  <0$ sur le support de $F$) sont orbitalement stables par
  perturbation \`a sym\'etrie sph\'erique pour le syst\`eme
  d'\'evolution \eqref{eq:1}-\eqref{eq:2}.
  \end{conj}
  \begin{conj} \label{conj2} Les galaxies $f^0= F(E)$ \`a
      sym\'etrie sph\'erique isotropes et d\'ecroissantes
    ($\partial_E F = F' <0$ sur le support de $F$) sont
    orbitalement stables par perturbation g\'en\'erale pour
    le syst\`eme d'\'evolution \eqref{eq:1}-\eqref{eq:2}.
\end{conj}

Les conjectures correspondantes pour le probl\`eme lin\'earis\'e ont
\'et\'e d\'emontr\'ees dans les ann\'ees 1960 et 1970 \`a la suite des
travaux fondateurs d'Antonov \cite{An1,An2}, mais leur r\'esolution
pour la dynamique non-lin\'eaire correcte vient d'\^etre obtenue par
Lemou, M\'ehats et Rapha\"el respectivement dans les articles
\cite{LMR8} et \cite{LMR9} (nous renvoyons \'egalement \`a \cite{GL}
pour une r\'eponse partielle importante sur la conjecture
\ref{conj1}). Nous nous concentrerons sur la r\'esolution de la
conjecture \ref{conj2} (article \cite{LMR9}), tandis que nous
renvoyons \`a \cite{LMR8} pour la r\'esolution de la conjecture
\ref{conj1}, qui implique les m\^emes id\'ees essentielles. Et lorsque
n\'ecessaire, nous illustrerons les m\'ethodes et r\'esultats avec les
deux mod\`eles principaux: mod\`eles polytropiques \eqref{eq:poly} et
mod\`ele de King \eqref{eq:king}. Ajoutons enfin qu'en l'absence de
sym\'etrie sph\'erique de la solution stationnaire, la question de la
stabilit\'e est ouverte, \emph{m\^eme pour le probl\`eme
  lin\'earis\'e}.

\section{\textup{Le probl\`eme lin\'earis\'e}}
\label{sec:le-probl-line}

\subsection{La dynamique lin\'earis\'ee}
\label{sec:la-dynam-line}

Il est possible d'adopter une vision g\'eom\'etrique de l'espace des
solutions du syst\`eme \eqref{eq:1}-\eqref{eq:2} en terme de structure
de Poisson, avec les invariants d\'efinissant une \emph{feuille
  symplectique}. Nous n'en aurons pas besoin dans cet expos\'e, et
nous renvoyons aux r\'ef\'erences \cite{We-i,YMC,MP1,MP2} et
\cite[p. 302]{AM} pour des approfondissements sur cet aspect. Nous
mentionnerons uniquement la notion de \emph{perturbation dynamiquement
accessible} qui a jou\'e un r\^ole dans le d\'eveloppement de la
th\'eorie perturbative.

Si l'on consid\`ere une solution stationnaire $f^0(x,v)$ on peut
lin\'eariser le syst\`eme \eqref{eq:1}-\eqref{eq:2} autour de cette
solution $f = f^0 + \varepsilon \, h$ pour obtenir l'\'equation
lin\'eaire suivante
\begin{equation}
  \label{eq:linVP}
  \left\{
  \begin{array}{l}
    \partial_t h + v \cdot \nabla_x h - \nabla_x \phi_{f^0} \cdot \nabla_v
    h = \nabla_x \phi_h \cdot \nabla_v f^0, \quad t \in \R_+, \ x \in
    \R^3, \ v \in \R^3,\vspace{0.3cm} \\ \displaystyle
    h(t=0,x,v) = h_0(x,v) ,\vspace{0.3cm} \\ \displaystyle
    \rho_h (t,x) := \int_{\R^3} h(t,x,v) \dd v \quad \mbox { et } \quad
    \phi_h(t,x) := - \frac{1}{4 \pi |x|} \ast \rho_h.
\end{array}
\right.
\end{equation}
Il est ais\'e de construire des solutions r\'eguli\`eres uniques et
r\'esoudre le probl\`eme de Cauchy pour ce syst\`eme d'\'equations aux
d\'eriv\'ees partielles lin\'eaires, par des m\'ethodes de point fixe et
en utilisant l'in\'egalit\'e d'interpolation fondamentale
de la proposition~\ref{prop:interpol-clef}.

On peut r\'e\'ecrire ce syst\`eme d'\'evolution de mani\`ere plus
compacte en 
\begin{equation}\label{eq:linVPbis}
 \left\{
  \begin{array}{l}
    \partial_t h + \left\{ h, E_{f^0} \right\} = 
    \nabla_x \phi_h \cdot \nabla_v f^0,\vspace{0.3cm} \\ \displaystyle
    h(t=0,x,v) = h_0(x,v).
\end{array}
\right. 
\end{equation}

Il est important de remarquer que la fluctuation $h$ n'est en
g\'en\'eral pas positive et que la structure hamiltonienne de
l'\'equation \eqref{eq:1} n'a pas surv\'ecu au processus de
lin\'earisation. Le syst\`eme lin\'earis\'e \eqref{eq:linVP} en
h\'erite toutefois la propri\'et\'e suivante. Si l'on d\'efinit une
perturbation initiale \emph{dynamiquement accessible} comme \'etant
cr\'e\'ee \`a partir de $f^0$ par un hamiltonien ext\'erieur donn\'e
$\bar H_{ext}$, on obtient formellement
\[
f = f^0 + \varepsilon \, h + O(\varepsilon^2) \quad \mbox{ et } \quad 
H = E_{f^0} + \varepsilon \, \bar H_{ext} \quad \mbox{ avec } \quad
h_{|t=0} =0. 
\]

Le premier ordre de lin\'earisation en $\varepsilon$ de l'\'equation
sur $h$ donne 
\[
\partial_t h + \left\{ f^0, \bar H_{ext} \right\} =
O(\varepsilon)
\]
et on obtient finalement $h_t \sim t \, \left\{ f^0, \bar H_{ext}
\right\}$ en temps petit. Ce dernier terme joue le r\^ole de la
direction de d\'erivation dans le processus de lin\'earisation. On
montre alors dans la proposition suivante que la forme $h_t = \{ g_t ,
f^0 \}$ est pr\'eserv\'ee au cours du temps pour l'\'equation
lin\'earis\'ee \eqref{eq:linVP} sur $h$, soit que l'on peut
restreindre l'\'equation lin\'earis\'ee \`a l'\og{}espace tangent\fg{}
des conservations hamiltoniennes, \textit{i.e.}, de la feuille
symplectique.

\begin{prop}
  Si $g$ est une solution (r\'eguli\`ere) de l'\emph{\'equation
  g\'en\'eratrice} suivante
\[
\left\{ 
\begin{array}{l} \displaystyle 
\partial_t g + \left\{ g, E_{f^0} \right\} = \phi_{\{ g, f^0\}},
\vspace{0.3cm} \\ \displaystyle
g_{|t=0} = g_{in}, 
\end{array}
\right.
\]
alors $h = \{ g, f^0 \}$ est une solution r\'eguli\`ere de
\eqref{eq:linVPbis}. 
\end{prop}

Remarquons d\`es lors, ce qui nous sera utile par la suite, qu'il est
\'equivalent de chercher une solution sous la forme $h = \{ g,f^0\}$
ou sous la forme $h = \{ g, E_{f^0}\}$: il suffit en effet de
multiplier $g$ dans la premi\`ere forme par $1/F'(E_{f^0})$ pour obtenir
la deuxi\`eme.

\begin{proof}
On calcule l'\'equation d'\'evolution sur $h_t := \{ g_t, f^0 \}$: 
\begin{multline*}
\partial_t h = \{ \partial_t g, f^0 \} = - \{ \{ g, E_{f^0} \}, f^0 \}
+ \{ \phi_h, f^0 \} \\
= \{ \{ E_{f^0}, f^0 \}, g \} + \{ \{ f^0, g\}, E_{f^0}\} +  \{ \phi_h, f^0 \}
\end{multline*}
o\`u l'on a utilis\'e l'identit\'e suivante sur le crochet de Poisson
\[
\{ \{ A, B \}, C \} + \{ \{ B, C \} , A \} + \{ \{ C, A, \}, B \}= 0.
\]
On utilise alors $\{ E_{f^0}, f^0\} =0$ (par d\'efinition d'une
solution stationnaire), et $\{ \phi_h, f^0\} = \nabla_x \phi_h \cdot
\nabla_v f^0$, ce qui permet de conclure \`a 
\[
\partial_t h + \{ h , E_{f^0} \} = \nabla_x \phi_h \cdot
\nabla_v f^0
\]
qui est bien l'\'equation recherch\'ee, avec la donn\'ee initiale
$h_{in} = \{ g_{in}, f^0\}$. 
\end{proof}

\subsection{\'Energie libre du probl\`eme lin\'earis\'e}
\label{sec:energ-libre-dant}

Dans le cas de mod\`eles sph\'eriques, l'article \cite{KO} a pour la
premi\`ere fois propos\'e une \emph{\'energie libre} pr\'eserv\'ee au
cours de l'\'evolution par le syst\`eme de Vlasov-Poisson
lin\'earis\'e dans le cas d'interactions de Coulomb, pour les
plasmas. Cette fonctionnelle sera ensuite adapt\'ee au cas
gravitationnel du syst\`eme \eqref{eq:linVP} par Antonov
\cite{An1,An2}, et ce dernier l'utilisera pour formuler un crit\`ere
de stabilit\'e.  Nous ne pr\'esentons ici que le cas isotrope
$f^0=F(E)$ pour simplifier l'exposition et nous rappelons la notation
$E=E_{f^0}(x,v)$.
\begin{prop}
La quantit\'e 
\[
\cF(h) = - \int_{\R^3 \times \R^3} \frac{|h|^2}{F'(E)} \dd x \dd v -
\int_{\R^3} |\nabla_x \phi_h|^2 \dd x
\] 
est conserv\'ee par les solutions r\'eguli\`eres du syst\`eme
\eqref{eq:linVP} dont le support est inclus dans le support de
$F'$. 
\end{prop}

 Le quotient $|h|^2/F'(E)$ doit ici \^etre compris comme \'etant
 \'egal \`a z\'ero lorsque les num\'erateurs et d\'enominateurs
 s'annulent. 

 \begin{proof}
   On souhaite diff\'erencier en temps l'expression de $\mathcal
   F(h)$. On calcule 
   \begin{multline*}
     \dt \int_{\R^3 \times \R^3} \frac{|h|^2}{F'(E)} \dd x \dd v =
     - 2 \, \drint \frac{\{ h, E \} \, h}{F'(E)} \ddxv
     + 2 \, \drint \frac{\Dx \phi_h \cdot \Dv f^0 \, h}{F'(E)}
     \ddxv \\
     = - 2 \, \drint \frac{\{ h, E \} \, h}{F'(E)} \ddxv
     + 2 \, \drint v \cdot \Dx \phi_h \, h 
     \ddxv
   \end{multline*}
puis on remarque que 
\[
\left\{ \frac{1}{F'(E)}, E \right\} = - \frac{F''(E)}{F'(E)^2}
\, \{ E, E \} = 0
\]
et ainsi par int\'egration par parties 
\[
2 \, \drint \frac{\{ h, E \} \, h}{F'(E)} \ddxv = \drint
\frac{\{ h^2, E \}}{F'(E)} \ddxv 
= - \drint h^2 \, \left\{ \frac{1}{F'(E)}, E \right\} \ddxv
=0
\]
d'o\`u 
\[
\dt \int_{\R^3 \times \R^3} \frac{|h|^2}{F'(E)} \dd x \dd v = 2
\, \drint v \cdot \Dx \phi_h \, h \ddxv.
\]

On calcule par ailleurs 
\begin{multline*}
\dt \rint |\Dx \phi_h |^2 \dd x = 2 \rint \Dx \dpt \phi_h \cdot \Dx
\phi_h \dd x = - 2 \rint \dpt \Delta_x \phi_h \, \phi_h \dd x \\
= - 2 \rint \dpt \rho_h \, \phi_h \dd x = - 2 \rint \left( \rint \dpt
  h dv \right) \, \phi_h \ddxv \\
= 2 \drint v \cdot \Dx h \phi_h \dd x \dd v = - 2 \drint v \cdot \Dx \phi_h
h \dd x \dd v.
\end{multline*}

La somme de ces deux derni\`eres \'equations fournit le r\'esultat. 
 \end{proof}

\subsection{Interpr\'etation variationnelle de l'\'energie libre} 

Quelle est l'origine de cette \'energie libre et surtout quel est son
lien avec le hamiltonien du probl\`eme non-lin\'eaire? Si l'on compare
le hamiltonien en deux fonctions diff\'erentes $f^0$ et $f$ on obtient
par un calcul \'el\'ementaire la formule
\begin{equation}\label{eq:diffHam}
\cH(f) = \cH(f^0) + \drint \left( \frac{|v|^2}{2} + \phi_{f^0} \right)
(f-f^0) \ddxv - \frac12 \rint |\Dx \phi_f - \Dx \phi_{f^0} |^2 \dd x.
\end{equation}
En effet 
\begin{multline*}
  \drint \left( \frac{|v|^2}{2} + \frac{\phi_f}{2} \right) f \ddxv =
  \drint \left( \frac{|v|^2}{2} + \frac{\phi_{f^0}}{2} \right) f^0
  \ddxv \\ + \drint \left( \frac{|v|^2}{2} + \phi_{f^0} \right)
  (f-f^0) \ddxv + \drint \left( \frac12 \phi_f f - \phi_{f^0} (f-f^0)
    - \frac12
    \phi_{f^0} f^0 \right) \ddxv \\
  = \cH(f^0) + \drint E (f-f^0) \ddxv - \rint \left( \frac12 |\Dx \phi_f|^2 -
    \Dx \phi_{f^0} \cdot \Dx \phi_f + \frac12 |\Dx \phi_{f^0}|^2 \right) \ddxv
\end{multline*}
d'o\`u le r\'esultat. 

Il est clair cependant qu'aucune fonction $f^0$ ne peut constituer un
point critique de $\cH$ \`a cause de la partie lin\'eaire en $(f-f^0)$
dans la formule \eqref{eq:diffHam} ci-dessus, qui n'est jamais
l'application nulle. Cela s'explique par les invariants du syst\`eme.
Pour y rem\'edier, une id\'ee introduite par Arnold \cite{Ar1,Ar2,Ar3}
dans les ann\'ees 1960 dans le cas de l'\'equation d'Euler
incompressible bidimensionnelle, est de consid\'erer une fonctionnelle
dite \emph{d'\'energie-Casimir} qui combine le hamiltonien et une
fonctionnelle de Casimir bien choisie. Cette id\'ee resurgira plus
loin lorsque nous aborderons la stabilit\'e non-lin\'eaire.

On se donne alors une fonctionnelle d'\'energie-Casimir g\'en\'erale
d\'efinie par une fonction $\cC : \R_+ \to \R_+$ avec $\cC(0)=0$:
\[
\cH_\cC(f) := \cH(f) + \drint \cC (f) \ddxv,
\] 
et l'on cherche \`a ajuster la fonction $\cC$ afin que cette
fonctionnelle admette bien $f^0$ comme point critique. Par un calcul
\'el\'ementaire et en d\'eveloppant $\cH_\cC$ autour de $f^0$ on
obtient
\begin{multline*}
\cH_\cC(f) = \cH_\cC(f^0) + \drint \left( \frac{|v|^2}{2} + \phi_{f^0}
  +\cC'(f^0) \right)
(f-f^0) \ddxv - \frac12 \rint |\Dx \phi_f - \Dx \phi_{f^0} |^2 \dd x\\ 
+ \frac12 \drint \cC''(f^0) (f-f^0)^2 \ddxv + \dots \\ =: \cH_\cC(f^0) +
D\cH_\cC(f^0)[f-f^0] + \frac12 D^2\cH_\cC(f^0)[f-f^0] + \dots
\end{multline*}

On voit alors que le choix formel 
\[
\cC'(f^0) = - E = - \left( \frac{|v|^2}{2} + \phi_{f^0} \right)
\]
permet d'annuler la partie lin\'eaire et d'obtenir exactement 
\[
D^2 \cH_\cC(f^0)[f-f^0] = \cF(f-f^0). 
\]
En effet puisque $f^0=F(E)$ ne d\'epend que de $E$ on a 
\[
\cC'(F(E)) = -E \quad \Longrightarrow \quad \cC''(F(E)) F'(E) =
-1 \ \mbox{ soit } \ \cC''(F(E)) = -\frac{1}{F'(E)}.
\]

On peut donc interpr\'eter cette \'energie libre comme la hessienne de
la fonctionnelle d'\'energie-Casimir au point $f^0$, pour une
contrainte $\cC$ correctement choisie en fonction de $f^0$.

\subsection{Premi\`ere approche na\"ive de la stabilit\'e lin\'eaire
  par interpolation}

 Donnons tout d'abord une approche \og na\"ive\fg{} de la stabilit\'e
 qui rappelle l'in\'egalit\'e d'interpolation de la proposition~\ref{prop:interpol-clef}
 pour l'\'equation non-lin\'eaire. On peut en effet interpr\'eter \`a
 nouveau la fonctionnelle $\cF$ comme une \'energie que l'on
 d\'ecompose en deux termes positifs:
\[
\cF(h) = \cF_{cin}(h) - \cF_{pot}(h) \quad \mbox{ avec }
\]
\[
\cF_{cin}(h) := - \int_{\R^3 \times \R^3} \frac{|h|^2}{F'(E)} \dd x
\dd v, \qquad \cF_{pot}(h) := \int_{\R^3} |\nabla_x \phi_h|^2 \dd x
\]
et l'on cherche \`a contr\^oler la partie $\cF_{pot}$ \`a partir de
$\cF_{cin}$ dans une estimation \emph{sous-critique}. La proposition
suivante est inspir\'ee de \cite{BMR}.
\begin{prop}[\cite{BMR}]
  Si la solution stationnaire $f^0$ v\'erifie 
\begin{equation}\label{eq:hypBMR}
\drint F'(E) \ddxv < + \infty
\end{equation}
alors 
\[
\cF_{pot} (h) \le C \| h \|_{L^1 (\R^6)} ^{2/3} \| \rho_h \|_{L^\infty (\R^3)}
^{1/3} \cF_{cin}(h)^{\frac12}
\]
et l'on peut d\'emontrer la stabilit\'e dans un espace de solutions
tel que $\rho_h \in L^\infty (\R^3)$. 
\end{prop}

La preuve de l'in\'egalit\'e fonctionnelle est imm\'ediate, et le
r\'esultat de stabilit\'e lin\'eaire en d\'ecoule une fois les
solutions correctement construites. Malheureusement l'hypoth\`ese
\eqref{eq:hypBMR} ne permet pas de traiter les mod\`eles physiques
les plus int\'eressants. Il va donc falloir se tourner vers une estimation
globale de $\cF$ sans d\'ecomposition; autrement dit, une estimation
de convexit\'e sur la fonctionnelle d'\'energie-Casimir $\cH_\cC$.

\subsection{L'in\'egalit\'e de coercitivit\'e d'Antonov}
\label{sec:lineg-de-coerc}

Antonov \cite{An1,An2} propose alors de consid\'erer la fonctionnelle
$\cF$ dans son ensemble, et \emph{de montrer qu'elle est positive, une
  fois restreinte aux perturbations dynamiquement accessibles} $h = \{
g, f^0 \}$. Pour \^etre exact, l'argument (formel) original d'Antonov
n'est pas fond\'e sur les perturbations dynamiquement accessibles,
mais sur une d\'ecomposition de la perturbation en partie paire et
partie impaire par rapport \`a la variable de vitesse. Nous renvoyons
\'egalement aux travaux ult\'erieurs dans la litt\'erature physique
qui ont d\'evelopp\'e l'id\'ee d'Antonov \cite{DFB,DF,SdFLRP,KS,PA}.

L'argument historique d'Antonov consiste \`a consid\'erer une solution
$h$ du probl\`eme lin\'earis\'e \eqref{eq:linVPbis} puis \`a la
d\'ecomposer en partie paire et partie impaire par rapport \`a la
variable de vitesse $v$, soit $h= h_1 + h_2$ avec 
\[
h_1(x,v) = \frac12 \left( h(t,x,v) + h(t,x,-v)
\right), \quad h_2(x,v) = \frac12 \left( h(t,x,v) - h(t,x,-v)
\right).
\]
On obtient ainsi formellement les deux \'equations coupl\'ees
suivantes
\[
\left\{ 
\begin{array}{l} \displaystyle 
\partial_t h_1 + \{ h_2, E \} =0, \vs \\ \ds
\partial_t h_2 + \{ h_1, E \} = \Dx \phi_{h_1} \cdot \Dv f^0.
\end{array}
\right.
\]
Puis en rempla\c cant le terme $h_1$ dans la deuxi\`eme \'equation, on
d\'eduit \emph{l'\'equation de pulsation} suivante sur $h_2$
\begin{equation}\label{eq:pulsation}
\partial^2 _{tt} h_2 = \{ \{ h_2, E \}, E \} - \Dx
\phi_{\{h_2,E \}} \cdot \Dv f^0.
\end{equation}
Par analogie avec l'\'equation diff\'erentielle $y'' = \lambda y$, on
voit que, si l'op\'erateur du membre de droite est n\'egatif, on
esp\`ere obtenir des solutions oscillantes born\'ees (ce qui justifie
le terme de \og pulsation\fg{}).

On calcule alors l'\'energie de ce syst\`eme.
\begin{prop}\label{prop:energie-pulsation}
  On suppose $F'<0$ sur son domaine, et $h \in C^\infty_c$
  une solution de \eqref{eq:pulsation} \`a support inclus dans celui
  de $F'$. Alors $h$ v\'erifie
\[
\td{}{t} \left( - \drint \frac{|h |^2}{F'(E)} \ddxv - \drint \frac{|\{
h, E \}|^2}{F'(E)} \ddxv - \rint |\Dx \phi_{\{h,f^0\}} |^2 \dd x
\right) =0
\]
soit 
\[
\td{}{t} \left( - \drint \frac{|h |^2}{F'(E)} \ddxv +
  \cF(\{h,f^0\}) \right) =0.
\]
\end{prop}

\begin{proof}
  En int\'egrant contre $\partial_t h /F'(E)$ l'\'equation de
  pulsation on obtient
\[
\drint \partial^2 _{tt} h \frac{\partial_t  h}{F'(E)} \ddxv 
= \dt \frac12 \drint \frac{|\partial_t h|^2}{F'(E)} \ddxv 
\]
pour le premier terme, puis
\begin{multline*}
  \drint \{ \{ h, E \}, E \} \frac{\partial_t h}{F'(E)} \ddxv \\ =
  - \drint \{ h, E \} \frac{\{ \partial_t h, E\} }{F'(E)} \ddxv =
  - \frac12 \td{}{t} \drint \frac{|\{ h, E \}|^2}{F'(E)} \ddxv
\end{multline*}
pour le deuxi\`eme terme o\`u l'on a utilis\'e une int\'egration par
parties et le fait que $\{ 1/F'(E), E \} =0$, et enfin pour le dernier
terme
\begin{multline*}
  - \drint \Dx \phi_{\{h,E \}} \cdot \Dv f^0 \frac{\partial_t
    h}{F'(E)} \ddxv =
  - \drint \{ \phi_{\{h,E\}}, E \} \partial_t h \ddxv \\
  = \drint \phi_{\{h,E\}} \{ \partial_t h , E \} \ddxv
  = \rint \phi_{\{h,E \}} \partial_t \rho_{\{h , E \}}
  \dd x \\ = - \rint \Dx \phi_{\{h,E \}} \cdot \Dx \partial_t
  \phi_{\{ h , E  \}} \dd x = - \frac12 \td{}{t} \rint |\Dx
  \phi_{\{h,E \}}|^2\dd x.
\end{multline*}
On obtient le r\'esultat souhait\'e en combinant les trois
pr\'ec\'edentes \'egalit\'es. 
\end{proof}

On voit donc que, d\`es lors que l'on sait montrer la positivit\'e de
l'\'energie libre sur les perturbations admissibles $\{ h_2, f^0\}$
issues d'une fonction $h_2$ impaire, on peut en d\'eduire
l'impossibilit\'e d'instabilit\'es sur la composante impaire $h_2$
puisque l'on a alors la somme de deux termes positifs qui est
conserv\'ee au cours du temps. C'est le crit\`ere de stabilit\'e
d'Antonov \cite{An1}. Le reste de l'argument de \cite{An1} est moins
clair en ce qui concerne le contr\^ole de la composante paire $h_1$,
et semble conditionn\'e \`a certaines hypoth\`eses \textit{ad hoc} pour
\'eviter la possibilit\'e d'instabilit\'es avec une croissance
polynomiale en temps. Dans l'article suivant \cite{An2}, Antonov
donne une preuve de cette propri\'et\'e cruciale de positivit\'e de
$\cF(\{ h_2,E \})$ avec $h_2$ impaire, en r\'esolvant le probl\`eme de
minimisation associ\'e.

Nous allons maintenant pr\'esenter les analyses math\'ematiques
r\'ecentes de ces questions. Tout d'abord, nous pr\'esentons un
\'enonc\'e et une preuve \'el\'egante issue de \cite{GR3} de la
coercitivit\'e d\'ecouverte par Antonov, \textit{i.e.}, la
positivit\'e de l'\'energie libre sur les fonctions de la forme $\{
h_2, f^0\}$ avec $h_2$ impaire. Du fait de l'importance de cette
propri\'et\'e d'un point de vue historique, ainsi que pour la gen\`ese
des r\'esultats dont il est question ici, nous pr\'esentons une preuve
d\'etaill\'ee. Une autre approche de cette propri\'et\'e et de sa
preuve, d\'evelopp\'ee dans \cite{LMR8,LMR9}, sera discut\'ee dans la
section \ref{sec:la-stabilite-non}.

\begin{prop}[\cite{GR3}]
  \label{prop:Antonov-simple}
  On suppose $F'<0$ sur le domaine de $f^0$, et $h \in C^\infty_c$
  \`a support inclus dans celui de $F'$, \`a sym\'etrie
  sph\'erique et impaire par rapport \`a la variable de vitesse. Alors
  on a
\begin{equation}\label{eq:Antonov-simple}
\cF(\{h,f^0\}) \ge - \drint F'(E) \left( |x \cdot v|^2
  \left| \left\{ \frac{h}{x\cdot v}, E \right\} \right|^2 +
  \frac{\phi_{f^0}'(r)}{r} |h|^2 \right) \ddxv \ge 0
\end{equation}
o\`u $r=|x|$ et $\phi_{f^0}'(r)$ d\'esigne la d\'eriv\'ee radiale de
$\phi_{f^0}=\phi_{f^0}(r)$. 
\end{prop}

\begin{proof}
  Observons tout d'abord que, pour un mod\`ele sph\'erique $f^0=F(E)$,
  les fonctions $\phi_{f^0}$ et $\rho_{f^0}$ sont radialement
  sym\'etriques, et on a 
\[
\Delta_x \phi_{f^0} = \rho_{f^0} \quad \Longrightarrow \quad
\frac{1}{r^2} \left[ r^2 \phi'_{f^0}(r) \right]' = \rho_{f^0}(r)
\]
d'o\`u, en int\'egrant de $0$ \`a $r$, on obtient 
\[
\fa r >0, \quad \phi'_{f^0}(r) = \frac{1}{r^2} \int_0 ^r r_* ^2
\rho_{f^0} (r_*) \dd r_* \ge 0
\]
et le potentiel cr\'e\'e par la solution stationnaire $f^0$ est
croissant. On d\'eduit donc imm\'ediatement la positivit\'e du membre
de droite dans la proposition. 

Rappelons la formule de l'\'energie libre 
\[
\cF(\{ h, f^0\}) := - \int_{\R^3 \times \R^3} \frac{|\{ h,
  f^0\}|^2}{F'(E)} \dd x \dd v - \int_{\R^3} |\nabla_x \phi_{\{h, f^0\}}|^2 \dd x.
\]
Nous allons tout d'abord contr\^oler par au-dessus l'oppos\'e du
second terme. On observe d'une part que la sym\'etrie sph\'erique de
$h$ implique imm\'ediatement celle de $\phi_{\{h,f^0\}}$, et d'autre
part que
\begin{multline*}
  \rint \{h, f^0\} \dd v = \rint \left( \Dx h \cdot \Dv f^0 - \Dv h
    \cdot \Dx
    f^0 \right) \dd v \\
  = \Dx \cdot \left( \rint h \Dv f^0 \dd v \right) = \Dx \cdot \left(
    \rint v F'(E) h \dd v \right).
\end{multline*}
Par cons\'equence, en utilisant \`a nouveau $\Delta_x \phi= r^{-2} [
r^2 \phi' ]'$, on obtient 
\begin{multline*}
\phi_{\{h, f^0\}}'(r) = \frac{1}{r^2} \int_0 ^r \left( \rint \{h(x_*,v),
  f^0(x_*,v)\} \dd v \right)  r_* ^2 \dd r_* 
\\
= \frac{1}{4 \pi r^2} \int_{|x|\le r} \Dx \cdot \left(
  \rint v F'(E) h \dd v \right) \dd x = \frac{1}{4 \pi r^2}
\int_{|x| = r} \rint \left( \frac{x \cdot v}{|x|} \right) F'(E) h \dd v \dd \omega(x)
\end{multline*}
o\`u $\mathrm d \omega$ d\'esigne la mesure uniforme sur la sph\`ere, et dans
la premi\`ere ligne $|x_*|=r_*$. Puis, l'int\'egrande est \`a nouveau
invariant par rotation sur la variable $x$ et on en d\'eduit
\[
\fa |x|=r, \quad \phi_{\{h, f^0\}}'(r) = \rint \left( \frac{x
    \cdot v}{|x|} \right) F'(E)  h \dd v.
\]
On aboutit au contr\^ole suivant sur le second terme de
l'\'energie libre par in\'egalit\'e de Cauchy-Schwarz:
\[
\int_{\R^3} |\nabla_x \phi_{\{h, f^0\}}|^2 \dd x \le \rint \left( -
  \rint \left( \frac{x \cdot v}{|x|} \right)^2 F'(E) \dd v \right)
\left( - \rint F'(E) |h|^2 \dd v\right) \dd x.
\]
\'Etudions plus pr\'ecis\'ement la premi\`ere int\'egrale en $v$ du
membre de droite. En chaque point $x$, on d\'ecompose orthogonalement
la variable $v$ en
\[
v = \left( \frac{x \cdot v}{|x|} \right) \frac{x}{|x|} + \left( \frac{x \wedge
    v}{|x|} \right) =: v_{||} \frac{x}{|x|} + v_{\bot}, \qquad
v_{||}\in \R, \ v_{\bot} \in \R^2
\]
et 
\begin{multline*}
  - \rint \left( \frac{x \cdot v}{|x|} \right)^2 F'(E) \dd v = -
  \int_{\R^2} \int_{\R} v_{||} ^2 F'\left( \frac{|v_{||}|^2}{2}
    + \frac{|v_{\bot}|^2}{2} + \phi_{f^0}\right) \dd v_{||} \dd v_{\bot}
  \\
  = -2 \int_{\R^2} \int_{\R_+} v_{||} ^2 F'\left(
    \frac{v_{||}^2}{2} + \frac{|v_{\bot}|^2}{2} + \phi_{f^0}(r)\right)
  \dd v_{||} \dd v_{\bot} \\
= 2 \int_{\R^2} \int_{\R_+} F\left(
    \frac{v_{||}^2}{2} + \frac{|v_{\bot}|^2}{2} + \phi_{f^0}(r)\right)
  \dd v_{||} \dd v_{\bot} = \drint f^0(x,v) \ddxv = \rho_{f^0}(r).
\end{multline*}
On a donc 
\[
\int_{\R^3} |\nabla_x \phi_{\{h, f^0\}}|^2 \dd x \le - \drint
 F'(E) \rho_{f^0}(r) |h|^2 \ddxv.
\]

On calcule maintenant le carr\'e du crochet de Poisson suivant
\begin{multline*}
  |\{ h, f^0 \}|^2 = \left| \left\{ \frac{h}{(x\cdot v)} (x \cdot v),
      f^0 \right\} \right|^2 = \left| (x \cdot v) \left\{
      \frac{h}{(x\cdot v)}, f^0 \right\} + \frac{h}{(x\cdot v)}
    \left\{ (x \cdot v),
      f^0 \right\} \right|^2 \\
  = (x \cdot v)^2\left| \left\{ \frac{h}{(x\cdot v)}, f^0 \right\}
  \right|^2 + \left( \frac{h}{(x\cdot v)} \right)^2 \left| \left\{ (x
      \cdot v), f^0 \right\} \right|^2 + 2 h \left\{ \frac{h}{(x\cdot
      v)}, f^0
  \right\} \left\{ (x \cdot v), f^0 \right\} \\
= (x \cdot v)^2\left| \left\{ \frac{h}{(x\cdot v)}, f^0 \right\}
  \right|^2 + \left\{ \left( \frac{h}{ (x\cdot v)}\right)^2 (x \cdot
    v) \{ (x \cdot v), f^0 \}, f^0 \right\} \\ 
 - \left( \frac{h^2}{(x \cdot v)} \right) \big\{ \{ (x \cdot v), f^0\},
 f^0\big\}
\end{multline*}
ce qui donne, lorsque l'on int\`egre contre $1/F'(E)$: 
\begin{multline*}
  \drint \frac{|\{ h, f^0 \}|^2}{F'(E)} \ddxv = \drint \frac{(x \cdot
    v)^2}{F'(E)} \left| \left\{ \frac{h}{(x\cdot v)}, f^0 \right\}
  \right|^2 \ddxv \\ + \drint \frac{1}{F'(E)} \left\{ \left( \frac{h}{
        (x\cdot v)}\right)^2 (x \cdot v) \{ (x \cdot v), f^0 \}, f^0
  \right\} \ddxv \\ - \drint \frac{1}{F'(E)} \left( \frac{h^2}{(x
      \cdot v)} \right) \big\{ \{ (x \cdot v), f^0\}, f^0\big\} \ddxv
  \\
  = \drint (x \cdot v)^2 F'(E) \left| \left\{ \frac{h}{(x\cdot v)}, E
    \right\} \right|^2 \ddxv \\ + \drint \left\{ \left( \frac{h}{
        (x\cdot v)}\right)^2 (x \cdot v) \{ (x \cdot v), f^0 \}, E
  \right\} \ddxv \\ - \drint F'(E) \left( \frac{h^2}{(x \cdot v)}
  \right) \big\{ \{ (x \cdot v), E\}, E\big\} \ddxv.
\end{multline*}
Le deuxi\`eme terme du membre de droite s'annule par int\'egration par
parties, et comme 
\begin{multline*}
\big\{ \{ (x \cdot v), E\}, E\big\} = - (x \cdot v ) \phi_{f^0}''(r)
- 3 \frac{(x \cdot v)}{r} \phi_{f^0}'(r) \\ = - \frac{(x \cdot v)}{r^2}
\left( \phi_{f^0} ''(r) r^2 + 2 \phi_{f^0}'(r) r \right) - \frac{(x
  \cdot v)}{r} \phi_{f^0}'(r) \\
= - \frac{(x \cdot v)}{r^2}
\left( r^2 \phi_{f^0} '(r) \right)' - \frac{(x
  \cdot v)}{r} \phi_{f^0}'(r) = - (x \cdot v) \rho_{f^0}(r) - \frac{(x
  \cdot v)}{r} \phi_{f^0}'(r),
\end{multline*}
on obtient pour le troisi\`eme terme 
\[
- \drint F'(E) \left( \frac{h^2}{(x \cdot v)} \right) \big\{ \{ (x \cdot
v), E\}, E\big\} \ddxv = \drint F'(E) h^2 \left(
  \rho_{f^0} + \frac{\phi_{f^0}'}{r} \right)  \ddxv
\]
et donc 
\begin{multline*}
- \drint \frac{|\{ h, f^0 \}|^2}{F'(E)} \ddxv 
  = - \drint (x \cdot v)^2 F'(E) \left| \left\{ \frac{h}{(x\cdot
        v)}, f^0 \right\} \right|^2 \ddxv \\ - \drint F'(E) h^2 \left(
  \rho_{f^0} + \frac{\phi_{f^0}'}{r} \right)  \ddxv.
\end{multline*}
On conclut la preuve en combinant cette \'egalit\'e avec
l'in\'egalit\'e pr\'ec\'edente sur $\int_{\R^3} |\nabla_x \phi_{\{h,
  f^0\}}|^2 \dd x$.
\end{proof}

\subsection{Une in\'egalit\'e de coercitivit\'e pr\'ecis\'ee \`a la
  Weinstein pour les polytropes}
\label{sec:une-inegalite-de}

Pour conclure cette section, nous allons maintenant consid\'erer le
mod\`ele polytropique \eqref{eq:poly} et pr\'esenter l'analyse
lin\'earis\'ee rigoureuse effectu\'ee dans \cite{LMR3}. Nous renvoyons
\`a la section suivante pour la caract\'erisation variationnelle des
polytropes de \cite{LMR4} qui est utilis\'ee pour montrer la
positivit\'e au sens large de la fonctionnelle d'\'energie libre du
probl\`eme lin\'earis\'e, et nous montrons comment les auteurs
de \cite{LMR3} en d\'eduisent une in\'egalit\'e de coercitivit\'e
pr\'ecis\'ee.

Nous pr\'esentons ce r\'esultat et une \'ebauche de preuve car cette
derni\`ere contient l'une des id\'ees de la m\'ethode non-lin\'eaire
de \cite{LMR9}: l'in\'egalit\'e de coercitivit\'e sur l'op\'erateur de
Schr\"odinger \eqref{eq:opSchr} ci-dessous, qui sera utilis\'ee pour
pouvoir traiter les perturbations non radiales dans le cas
non-lin\'eaire. Ce travail est inspir\'e de l'\'etude par Weinstein
\cite{We-ii-1,We-ii-2} dans les ann\'ees 1980 de la stabilit\'e des
solitons pour l'\'equation de Schr\"odinger.

On consid\`ere une solution stationnaire polytropique \eqref{eq:poly}
avec $n = 1/(p-1)$ et $E_0=-1$: 
\begin{equation}\label{eq:polyp}
f^0 = F(E) = \left( -1 - E \right)_+ ^{1/(p-1)} =
\left( -1 - \frac{|v|^2}{2} - \phi_{f^0} \right)_+ ^{1/(p-1)}
\end{equation}
pour $p > p_c = 9/7$ sup\'erieur strictement \`a l'exposant critique
que nous avons d\'ej\`a rencontr\'e dans l'in\'egalit\'e
d'interpolation de la proposition~\ref{prop:interpol-clef}, ce qui
redonne exactement la condition $n <7/2$. On d\'efinit l'op\'erateur
\[
\cM h = \left( - \frac{h}{F'(E)} + \phi_h \right) 1_K
\]
restreint au support $K$ de $f^0$. 
Il est clair que 
\[
\langle \cM h , h \rangle_{L^2(\R^3 \times \R^3)} = \cF(h)
\]
et que cet op\'erateur lin\'eaire est sym\'etrique dans $L^2(K,\dd \mu)$
avec la mesure de r\'ef\'erence $\mu = 1/|F'(E)|$. On peut donc
reformuler la question de contr\^oler par en-dessous l'\'energie libre
en un \emph{contr\^ole de coercitivit\'e} sur l'op\'erateur $\cM$.

Tout d'abord, nous admettons ici la positivit\'e (au sens large) $\cM
\ge 0$ de cet op\'erateur pour des perturbations \emph{qui ne
  modifient pas le hamiltonien}:
\[
h \in L^2(K,d\mu) \quad \mbox{ avec } \quad \drint h E \ddxv = \drint
h \left( \frac{|v|^2}{2} + \phi_{f^0} \right) \ddxv = 0.
\]
Cette positivit\'e d\'ecoule de la caract\'erisation variationnelle de
la proposition~\ref{prop:var-lmr-lions} par saturation
  d'in\'egalit\'e de Sobolev de la section suivante. 

  Une fois cette positivit\'e acquise, en s'inspirant de
  \cite{We-ii-1}, Lemou, M\'ehats et Rapha\"el quantifient la
  coercitivit\'e de l'\'energie libre de la mani\`ere suivante.

\begin{prop}\label{prop:weinsteinlin}
  Pour $p_c <p <+\infty$, la forme quadratique $h \mapsto \langle \cM
  h, h\rangle$ est continue et auto-adjointe sur $L^2(K,\dd \mu)$, et il
  existe une constante $\delta$ ne d\'ependant que de $p$ telle que
  \begin{equation}
    \label{eq:lin-coerc}
    \langle \cM h, h \rangle \ge \delta \int_K \frac{h^2}{|F'(E)|} -
    \frac{1}{\delta} \left[ \left( \drint h E \ddxv \right) + \sum_{i=1} ^3 \left(
        \drint x_i h \ddxv \right)^2 \right].
  \end{equation}
\end{prop}

On retrouve le r\'esultat pr\'ec\'edent (avec cependant une constante
non constructive) lorsque $h = \{ g, E \}$ et $g$~impaire en $v$, mais
l'analyse du d\'efaut de coercitivit\'e pour les perturbations qui ne
sont pas sous cette forme est ici plus pr\'ecise.

\begin{proof}[\'Ebauche de preuve]
On va raisonner par l'absurde. On d\'efinit $\cZ$ l'ensemble des
fonctions $h \in L^2(K,d\mu)$ telles que 
\[
\drint h E \ddxv = \drint x_1 h \ddxv  = \drint x_2 h \ddxv = \drint
x_3 h \ddxv = 0.
\]
On sait par la caract\'erisation variationnelle de la
proposition~\ref{prop:var-lmr-lions} que $\langle \cM h, h \rangle \ge
0$ pour $h \in \cZ$. On suppose alors que
\[
I := \inf_{h \in \cZ, \ \| h\|_{L^2(K,d\mu)} =1} \langle \cM h, h \rangle
  = 0.
\]
\sk

\noindent \'Etape 1 -- \emph{Construction d'un minimiseur}. On
consid\`ere une suite minimisante $h_n \in \cZ$. Par compacit\'e
faible, quitte \`a extraire, la suite $h_n$ converge faiblement vers
une fonction $h_\infty$ dans $L^2(K,\dd \mu)$. Par ellipticit\'e de
l'\'equation de Poisson, on a la convergence forte de $\Dx \phi_{h_n}$
dans $L^2 (\R^3)$. De par la normalisation, on a
\[
\langle \cM h_n, h_n \rangle = 1 - \| \Dx \phi_{h_n} \|_{L^2 (\R^3)} ^2 \to 0 
\]
d'o\`u $ \| \Dx \phi_{h_\infty} \|_{L^2 (\R^3)}=1$ et la limite $h_\infty$ n'est pas
nulle.  \sk

\noindent \'Etape 2 -- \emph{Le minimiseur est dans le noyau de
  $\cM$}. Par la m\'ethode des multiplicateurs de Lagrange, on
d\'eduit du probl\`eme de minimisation sous contraintes ci-dessus que
\[
\cM h_\infty \in \mbox{Vect} \left\{ \frac{h_\infty}{|F'(E)|}, x_1
  1_K, x_2 1_K, x_3 1_K, E 1_K \right\}.
\]
En consid\'erant les int\'egrations de $\cM h_\infty$ contre
successivement $h_\infty$, $\partial_{x_1} f^0$, $\partial_{x_2} f^0$,
$\partial_{x_3} f^0$ et $x \cdot \Dv f^0 - 2 v \cdot \Dv f^0$, on
annule chacun des coefficients dans la d\'ecomposition selon cette
famille vectorielle, et on d\'eduit finalement que $\cM h_\infty =0$. 

\noindent \'Etape 3 -- \emph{\'Etude du noyau de $\cM$}. On montre que 
\[
\mbox{Ker} \cM = \mbox{Vect} \left\{ \partial_{x_1} f^0, \partial_{x_2} f^0,
\partial_{x_3} f^0 \right\}.
\]
L'inclusion de ces trois vecteurs dans le noyau de $\cM$ provient
simplement de la d\'erivation par rapport \`a $x_1$, $x_2$ et $x_3$ de
l'\'equation d\'efinissant la solution stationnaire 
\begin{equation}\label{eq:diffstat}
0 = \partial_{x_i} \left( (f^0)^{p-1} + \frac{|v|^2}2 + \phi_{f^0} + 1
\right) 1_K 
= \left( \frac{1}{F'(E)} \partial_{x_i} f^0 + \phi_{\partial_{x_i}f^0}
\right) 1_K.
\end{equation}

Pour l'inclusion r\'eciproque, on consid\`ere l'\'equation 
\begin{equation}\label{eq:stath}
\frac{h}{F'(E)} = \phi_h
\end{equation}
sur $K$, et l'on en d\'eduit l'\'equation \emph{r\'eduite} suivante
sur le potentiel par int\'egration en vitesse 
\[
\Delta_x \phi_h = \rho_h = \rint h \ddxv = \drint F'(E) \phi_h \ddxv. 
\]
Si l'on d\'efinit le potentiel effectif 
\[
V_{f^0} := - \rint F'(E) \dd v,
\]
on fait donc appara\^itre une \emph{\'equation de Schr\"odinger}
stationnaire
\begin{equation}\label{eq:opSchr}
\cA \phi_h := \left( \Delta_x  + V_{f^0} \right) \phi_h =0.
\end{equation}
Le noyau de cet op\'erateur $\cA$ est alors \'etudi\'e de mani\`ere
fine par d\'ecomposition selon les harmoniques sph\'eriques, avec pour
r\'esultat
\[
\mbox{Ker}(\cA) = \mbox{Vect}\left\{ \partial_{x_1}
  \phi_{f^0}, \partial_{x_2} \phi_{f^0}, \partial_{x_3} \phi_{f^0}
\right\}.
\]
On d\'eduit donc que $\phi_h \in \mbox{Vect}\left\{ \partial_{x_1}
  \phi_{f^0}, \partial_{x_2} \phi_{f^0}, \partial_{x_3} \phi_{f^0}
\right\}$, puis en utilisant \eqref{eq:stath} et \eqref{eq:diffstat}
que $h \in \mbox{Vect}\left\{ \partial_{x_1} f^0, \partial_{x_2}
  f^0, \partial_{x_3} f^0 \right\}$.

\noindent \'Etape 4 -- \emph{Conclusion}. Puisque $h \in \cZ$, en
int\'egrant la relation lin\'eaire $h \in
\mbox{Vect}\left\{ \partial_{x_1} f^0, \partial_{x_2}
  f^0, \partial_{x_3} f^0 \right\}$ contre $x_1$, $x_2$ et $x_3$, on
obtient successivement que tous les coefficients de la combinaison
lin\'eaire sont nuls, soit $h=0$, ce qui aboutit \`a une
contradiction. 
\end{proof}

Mentionnons qu'au moyen de cette in\'egalit\'e de coercitivit\'e,
Lemou, M\'ehats et Rapha\"el d\'emontrent ensuite dans \cite{LMR3} un
th\'eor\`eme de stabilit\'e lin\'earis\'ee qui \'enonce que, pour
toute donn\'ee initiale dans un \og espace d'\'energie\fg{}
(correspondant au probl\`eme de minimisation non-lin\'eaire), le
semi-groupe lin\'earis\'e cro\^it au plus en $O(t^2)$. Ils donnent
\'egalement une d\'ecomposition de l'espace $L^2(K,d\mu)$ qui localise
plus pr\'ecis\'ement les modes de croissance alg\'ebrique.

\section{La stabilit\'e non-lin\'eaire}
\label{sec:la-stabilite-non}

Nous allons maintenant suivre le cheminement des diff\'erents travaux
sur la stabilit\'e non-lin\'eaire. En dehors du dernier travail
\cite{LMR9} que nous d\'etaillerons, nous donnons seulement les
\'etapes principales.

\subsection{Premi\`eres approches variationnelles}
\label{sec:appr-par-minim}

Le grand succ\`es de ce programme conduit par diff\'erents groupes
ind\'ependants est la preuve de la stabilit\'e de tous les mod\`eles
polytropiques discut\'es pr\'ec\'edemment.  Les ingr\'edients communs
\`a ces diff\'erentes approches sont:
\begin{enumerate}
\item[(a)] la caract\'erisation de la solution stationnaire $f^0$
  \'etudi\'ee comme un \emph{\'etat fondamental} (\og ground
  state\fg{}) d'un probl\`eme de minimisation de la forme
\begin{equation}\label{eq:minabs}
\min_{\cV(f) = \mbox{{\scriptsize cstes}}} \ \cU(f) 
\end{equation}
pour une certaine fonctionnelle $\cU$ et un ensemble de contraintes
$\cV$, qui \emph{sont pr\'eserv\'ees par l'\'evolution non-lin\'eaire};
\item[(b)] la preuve d'une \emph{propri\'et\'e de s\'eparation des
    \'etats fondamentaux}: la solution stationnaire $f^0$ est isol\'ee
  parmi les minimiseurs du probl\`eme de minimisation pr\'ec\'edent;
\item[(c)] la \emph{compacit\'e des suites minimisantes}, qui repose
  souvent sur une technique de concentration-compacit\'e
  \cite{Li1,Li2}, sachant que la compacit\'e doit \^etre obtenue dans
  un sens assez fort pour permettre d'obtenir \`a la limite une
  solution stationnaire du syst\`eme;
\item[(d)] la preuve de la stabilit\'e est alors fond\'ee sur un
  raisonnement par l'absurde: on consid\`ere une suite minimisante qui
  ne reste pas proche de la solution stationnaire $f^0$, puis en
  appliquant (c) on aboutit, \`a la limite, \`a un \'etat stationnaire qui
  minimise le probl\`eme \eqref{eq:minabs} mais qui est diff\'erent de
  $f^0$, ce qui contredit (a)-(b).
\end{enumerate}

L'id\'ee d'utiliser une fonctionnelle d'\'energie-Casimir bien choisie
et d'\'etudier ses points critiques et sa convexit\'e a \'et\'e
introduite pour les \'equations d'Euler incompressibles en dimension
$2$ par Arnold \cite{Ar1,Ar2,Ar3}, et elle a ensuite \'et\'e
appliqu\'ee avec succ\`es aux \'equations de Vlasov-Poisson pour les
plasmas dans \cite{HMRW,Re1} (voir \'egalement les r\'ef\'erences
incluses dans \cite{HMRW} pour les travaux ant\'erieurs de physique
sur la stabilit\'e formelle pour les plasmas, ainsi que \cite{Gu1,Gu2}
pour des travaux math\'ematiques dans le cas de plasmas
magn\'etiques).

En ce qui concerne le syst\`eme de Vlasov-Poisson gravitationnel, le
premier travail pr\'ecurseur en ce sens est d\^u \`a Wolansky
\cite{Wo-i}: ce dernier caract\'erise l'\'equilibre comme le
minimiseur d'une fonctionnelle d'\'energie-Casimir
\[
\min_{\| f\|_{L^1} = 1 \ \& \ f \ge 0} \cH_\cC(f) \quad \mbox{avec}
\quad \cH_\cC(f) = \cH(f) + \drint \cC(f) \ddxv.
\]
Cette approche a \'et\'e d\'evelopp\'ee de mani\`ere syst\'ematique
par Guo et Rein \cite{Gu3,GR1,Re2} et a permis d'obtenir la
stabilit\'e des polytropes \eqref{eq:poly} pour $0<n \le 3/2$. Les cas
$3/2 < n \le 7/2$ ont ensuite \'et\'e \'etudi\'es dans
\cite{Gu4,GR2,RG,Sc2,Ha}. En particulier les articles \cite{Gu4,GR2}
modifient le probl\`eme variationnel de la fa\c con suivante:
\[
\min \cH(f) \quad \mbox{sous les contraintes} \quad f \ge 0, \ \frac{n}{n+1} \| f
\|_{L^{1+1/n}(\R^6)} ^{1+1/n} + \left( \frac72 - n\right) \|
f\|_{L^1(\R^6)} = M.
\]

Parall\`element et de mani\`ere l\'eg\`erement diff\'erente \`a ces
travaux, Dolbeault, S\'anchez et Soler \cite{DSS} introduisent ensuite
un probl\`eme de minimisation diff\'erent du type
\[
\min \cH(f) \quad \mbox{sous les contraintes} \quad f \ge 0, \ \| f
\|_{L^1 (\R^6)} = M, \ \| f\|_{L^\infty (\R^6)} \le 1.
\]
Ce probl\`eme de minimisation se r\'e\'ecrit de mani\`ere
\'equivalente et naturelle par saturation d'in\'egalit\'e
fonctionnelle de type Poincar\'e reliant l'\'energie potentielle,
l'\'energie cin\'etique, et les normes utilis\'ees: 
\[
\min \frac{\cH_{cin}(f)}{\cH_{pot}(f)} \quad \mbox{sous les
  contraintes} \quad f \ge 0, \ \| f \|_{L^1 (\R^6)} = M, \ \|
f\|_{L^\infty (\R^6)} \le 1.
\]
Cette approche a permis de traiter le cas formel limite $n=0$ dans
\eqref{eq:poly} (en plus de cas de polytropes anisotropes que nous
n'\'evoquons pas ici). S\'anchez et Soler \cite{SS} ont ensuite
g\'en\'eralis\'e cette approche \`a un espace de contrainte $\| f
\|_{L^1 (\R^6)} = M$ et $\| f\|_{L^p (\R^6)} \le 1$, et ont pu montrer
la stabilit\'e au sens de la distance $L^1 (\R^6)$ pour les polytropes
\eqref{eq:poly} avec $0 \le n < 7/2$. \footnote{Nous renvoyons
  \'egalement au travail \cite{CSS} qui \'etudie selon une strat\'egie
  proche les propri\'et\'es de stabilit\'e orbitale pour l'\'equation
  de Nordstr\"om-Vlasov dans un cadre relativiste.}  Un des apports de
ces travaux semble conceptuel: montrer que le probl\`eme variationnel
sous-jacent est reli\'e \`a des in\'egalit\'es de type Sobolev
optimales.

Simultan\'ement, Lemou, M\'ehats et Rapha\"el \cite{LMR1,LMR4}
caract\'erisent les polytropes \`a partir d'une in\'egalit\'e de type
Sobolev optimale correspondant aux in\'egalit\'es d'interpolation de
l'\'equation non-lin\'eaire. Le probl\`eme de minimisation en terme de
fonctionnelle d'\'energie et d'espace de contraintes est \'equivalent
\`a celui consid\'er\'e par S\'anchez et Soler. Mais ils font ainsi le
lien avec l'in\'egalit\'e d'interpolation de la
proposition~\ref{prop:interpol-clef}, et ils effectuent aussi un
retour conceptuel \`a la m\'ethode originelle de Cazenave et Lions
\cite{CL} pour l'\'etude de la stabilit\'e des solitons par
concentration-compacit\'e pour l'\'equation de Schr\"odinger.

Ils d\'emontrent la proposition suivante, que nous avons d\'ej\`a
\'evoqu\'ee et utilis\'ee dans l'\'etude lin\'earis\'ee pour la preuve
de la proposition~\ref{prop:weinsteinlin}.
\begin{prop}[\cite{LMR1,LMR4}]\label{prop:var-lmr-lions}
  Soient $p \in ]p_c, +\infty[$ et $f^0$ un polytrope d\'efini par
  \eqref{eq:polyp}. Alors le probl\`eme de minimisation 
\[
\min_{f \in \cE, \ f \not = 0} \frac{ \| |v|^2 f \|_{L^1 (\R^6)}^{\theta_1}
  \| f \|_{L^p (\R^6)} ^{\theta_2} \| f \|_{L^1 (\R^6)} ^{\theta_3}}{\| \Dx \phi_f
  \|^2 _{L^2 (\R^3)}} \quad \mbox{avec} \quad \theta_1 =\frac12, \ \theta_2
= \frac{p}{3 (p-1)}, \ \theta_3 = \frac{(7p-9)}{6(p-1)}
\]
(le dernier coefficient est bien positif du fait que $p > p_c =
9/7$) est atteint sur la famille \`a quatre param\`etres 
\[
\gamma f^0\left(\frac{x-x_0}{\lambda}, \mu v \right), \quad \gamma \in
\R_+ ^*, \ \lambda \in \R_+^*, \ \mu \in \R_+^*, \ x_0 \in \R^3.
\]
\end{prop}

En proc\'edant selon les grandes lignes de la strat\'egie d\'ecrite
plus haut, les auteurs d\'emontrent ensuite dans \cite{LMR4} la
stabilit\'e des polytropes \eqref{eq:poly} pour $0 < n < 7/2$.

Cette approche variationnelle a pu traiter de mani\`ere satisfaisante
les mod\`eles polytropiques. Elle semblait cependant impuissante \`a
traiter des mod\`eles plus g\'en\'eraux, et en particulier le mod\`ele
de King. La difficult\'e est que pour les types de probl\`emes de
minimisation que l'on vient d'\'enum\'erer, la propri\'et\'e de
s\'eparation des \'etats fondamentaux (b) d\'ecrite plus haut n'est en
g\'en\'eral plus v\'erifi\'ee, sans m\^eme parler de la propri\'et\'e
(c) de compacit\'e des suites minimisantes.

\subsection{\textup{Approche directe non-variationnelle par
    lin\'earisation}}
\label{sec:appr-non-vari}

Il existe essentiellement deux mani\`eres d'aborder la question de la
stabilit\'e d'un syst\`eme d'\'evolution non-lin\'eaire: d'une part
l'approche variationnelle o\`u l'on exprime la solution stationnaire
comme solution d'un probl\`eme de minimisation qui est invariant le
long de l'\'evolution, et d'autre part l'approche \og directe\fg{} par
lin\'earisation et contr\^ole du reste. Nous allons maintenant parler
des travaux s'inscrivant dans cette derni\`ere approche.

Dans le cas de cette approche directe, il faut tout d'abord quantifier
pr\'ecis\'ement les propri\'et\'es de stabilit\'e du syst\`eme
lin\'earis\'e. L'in\'egalit\'e de coercitivit\'e d'Antonov
\eqref{eq:Antonov-simple} est un point de d\'epart naturel pour
cela. Cependant il faut ensuite surmonter deux difficult\'es
importantes:
\begin{itemize}
\item il faut contr\^oler les termes d'ordre sup\'erieur ou \'egal \`a
  trois dans le d\'eveloppement de Taylor de la fonctionnelle
  d'\'energie-Casimir au voisinage de la solution stationnaire $f^0$
  consid\'er\'ee;
\item et l'autre difficult\'e est que l'in\'egalit\'e de coercitivit\'e
  d'Antonov \eqref{eq:Antonov-simple} n'est valide que pour les
  perturbations dynamiquement accessibles de la forme $h = \{ g,
  f^0\}$, et que l'on souhaiterait s'affranchir de cette restriction. 
\end{itemize}

La premi\`ere tentative d'utiliser cette approche remonte \`a Wan
\cite{Wa1} mais la preuve semble incompl\`ete. La deuxi\`eme tentative
du m\^eme auteur \cite{Wa2} est plus aboutie mais semble reposer sur
une hypoth\`ese non r\'ealiste de positivit\'e de la fonctionnelle
$\cF$ pour toute perturbation $h$ qui exclut la plupart des mod\`eles
physiques.

Le premier article traitant du mod\`ele de King, et suivant cette
approche directe, est d\^u \`a Guo et Rein \cite{GR3}. Les
ingr\'edients cl\'es de ce travail sont:
\begin{itemize}
\item la d\'efinition d'une classe de perturbation \`a sym\'etrie
  sph\'erique 
\begin{multline*}
\mathcal S_{f^0} = \Big\{ f \in L^1(\R^6), \ f=f(E,L) \ge 0 \ \mbox{
  tel que}\\ 
\fa \cC \in C^2(\R_+^2) \mbox{ avec } \cC(0,L) \equiv \partial_1
\cC(0,L) \equiv 0, \ \partial_1 ^2 \cC \mbox{ born\'e, alors } \\
\drint \cC(f,L) \ddxv = \drint \cC(f^0,L) \ddxv \Big\}
\end{multline*}
qui est: d'une part \emph{stable par le syst\`eme d'\'evolution
  non-lin\'eaire} du fait que $L$ est conserv\'e le long des
trajectoires et donc les fonctionnelles $\drint \cC(f,L) \ddxv$ sont
des invariants du syst\`eme et, d'autre part, incluse dans les
perturbations de la forme $h = \{ g, f^0\}$ (ce dernier point est
d\'emontr\'e en r\'esolvant le syst\`eme diff\'erentiel ordinaire
associ\'e);
\item la d\'emonstration par contradiction d'une propri\'et\'e de
  convexit\'e stricte de la fonctionnelle d'\'energie-Casimir $\cH$ au
  voisinage de $f^0$: 
\[
\cH(f) - \cH(f^0) \ge C_0 \| \Dx \phi_f - \Dx \phi_{f^0}
\|_{L^2(\R^3)} ^2
\] 
pour une constante $C_0 >0$ et avec 
\[
d(f,f^0) := \cH(f) - \cH(f^0) + \| \Dx \phi_f - \Dx \phi_{f^0}
\|_{L^2(\R^3)} ^2 \quad \mbox{ assez petit.}
\]
\end{itemize}
Il y a deux limitations importantes dans les r\'esultats
ainsi obtenus. D'une part les perturbations consid\'er\'ees sont
restreintes \`a la classe $\cS_{f^0}$, qui est \og trop petite\fg{},
et en particulier incluses dans l'ensemble des fonctions
\'equimesurables \`a $f^0$. D'autre part, on aimerait bien s\^ur
s'affranchir totalement de la contrainte de sym\'etrie sph\'erique
pour ces perturbations.

La premi\`ere de ces limitations a ensuite \'et\'e lev\'ee dans le
travail \cite{GL} de Guo et Lin. Ils d\'emontrent ainsi la stabilit\'e
du mod\`ele de King par petite perturbation \`a sym\'etrie
sph\'erique. Les \'el\'ements principaux de leur travail sont les
suivants. 
\begin{itemize}
\item Ils consid\`erent la fonctionnelle d'\'energie-Casimir 
\[
\cH_\cC(f) = \cH(f) + \drint \cC(f) \ddxv \quad \mbox{ avec } \quad 
\cC(f) := (1+f) \ln (1+f) -1 - f
\]
et la distance associ\'ee 
\[
d_\cC(f,f^0) := \cH_\cC(f) - \cH_\cC(f^0) + \| \Dx \phi_f - \Dx \phi_{f^0}
\|_{L^2(\R^3)} ^2.
\]
\item Afin de montrer la coercitivit\'e de cette fonctionnelle
  d'\'energie-Casimir au voisinage de $f^0$ sans faire appara\^itre de
  termes d'ordre sup\'erieur, ils font appel \`a une in\'egalit\'e de
  dualit\'e convexe \'el\'ementaire mais astucieusement utilis\'ee,
  qui permet de contr\^oler par en-dessous 
\begin{multline*}
  \cH_\cC(f) - \cH_\cC(f^0)  \ge \mbox{cste} \, \left( \rint |\Dx
    \phi_{f-f^0}|^2 \dd x +
    \drint F'(E) | \phi_{f-f^0} 
    - \mathcal P \phi_{f-f^0} |^2 \ddxv\right)\\ - \rint | \mathcal P
  \phi_{f-f^0} |^2 \dd x \ - \dots
\end{multline*}
o\`u les trois points d\'esignent des termes contr\^olables par les
invariants du syst\`eme, et l'op\'erateur $\mathcal P$ est le
projecteur sur le noyau de $D:=v \cdot \Dx - \Dx \phi_{f^0} \cdot \Dv$
(c'est un op\'erateur de moyennisation sur chacun des tores invariants
du flot compl\`etement int\'egrable associ\'e \`a cet op\'erateur de
transport).
\item Malheureusement le terme n\'egatif $-\rint | \mathcal P
  \phi_{f-f^0} |^2 \dd x$ dans l'\'equation ci-dessus semble
  difficilement contr\^olable, aussi les auteurs ont-ils l'id\'ee
  d'approcher l'op\'erateur $\mathcal P$ par une suite bien
  construite d'op\'erateurs $\mathcal P_\ell$ de rang fini, et de
  remplacer $\mathcal P$ par $\mathcal P_\ell$ dans l'argument
  ci-dessus. Le terme n\'egatif $-\rint | \mathcal P_\ell \phi_{f-f^0}
  |^2 \dd x$ est alors facilement contr\^olable en utilisant un nombre
  fini de fonctionnelles de Casimir invariantes.
\item Enfin il reste \`a \'etudier la coercitivit\'e du terme 
\[
\left( \rint |\Dx \phi_{f-f^0}|^2 \dd x +
\drint F'(E) | \phi_{f-f^0} - \mathcal P \phi_{f-f^0} |^2 \ddxv
\right).
\]
Cette derni\`ere implique sans mal une estimation de coercitivit\'e du
m\^eme terme avec $\mathcal P_\ell$ \`a la place de $\mathcal P$, pour
$\ell$ assez grand. Cela revient \`a \'etudier la positivit\'e de
l'op\'erateur 
\[
\cA \phi = - \Delta_x \phi + \rint F'(E) (\phi-\mathcal P \phi)\dd v
\]
agissant uniquement sur le potentiel $\phi$. En remarquant simplement
que 
\[
\mbox{Im}(1-\mathcal P) \, \bot \, \mbox{Ker}(D) \quad \mbox{ et } \quad \mbox{Ker}(D) =
\mbox{Im}(D)^\bot,
\]
on d\'eduit facilement que $(\phi-\mathcal P \phi) = \{ h, f^0\}$ avec
$h$ impaire, ce qui permet d'appliquer l'in\'egalit\'e de
coercitivit\'e d'Antonov \eqref{eq:Antonov-simple}, et de conclure.
\end{itemize}
Ce travail int\'eressant semble pouvoir se g\'en\'eraliser \`a des
mod\`eles sph\'eriques d\'ecroissants plus g\'en\'eraux que le
mod\`ele de King. La seconde limitation de cette m\'ethode,
c'est-\`a-dire le fait de ne consid\'erer que des perturbations \`a
sym\'etrie sph\'erique, semble par contre plus s\'ev\`ere. Un des
apports principaux du travail \cite{LMR9}, que nous allons maintenant
discuter, est de s'\^etre affranchi de cette limitation.

\subsection{\textup{Nouvelle approche variationnelle par
    r\'earrangement}}
\label{sec:nouv-appr-vari}

Apr\`es ce d\'etour par une approche non-variationnelle, nous allons
maintenant revenir \`a une approche variationnelle, mais sous un angle
nouveau. On voit qu'un d\'efaut de l'approche variationnelle par
\'energie-Casimir est qu'elle semble impuissante \`a reformuler sous
forme de probl\`eme de minimisation certains mod\`eles
stationnaires d\'ecroissants. Cependant, dans le m\^eme temps,
l'approche directe par lin\'earisation semble limit\'ee par
l'in\'egalit\'e de coercitivit\'e d'Antonov elle-m\^eme et par les
difficult\'es inh\'erentes aux contr\^oles des termes d'ordre
sup\'erieur dans le d\'eveloppement du hamiltonien.

Le travail \cite{LMR7} constitue une premi\`ere avanc\'ee en
introduisant l'id\'ee d'exploiter les propri\'et\'es
d'\'equimesurabilit\'e du flot: m\^eme si la propri\'et\'e de
s\'eparation des \'etats fondamentaux n'est pas v\'erifi\'ee pour le
probl\`eme de minimisation avec un nombre fini de contraintes,
l'\'equimesurabilit\'e de la solution \`a sa donn\'ee initiale permet
de prouver dans certains cas une propri\'et\'e de s\'eparation
\emph{locale}. Finalement dans les travaux \cite{LMR8,LMR9}, Lemou,
M\'ehats et Rapha\"el r\'esolvent compl\`etement ces
contradictions. Ils prouvent le th\'eor\`eme suivant dans le cas de
mod\`eles sph\'eriques isotropes.
\begin{theo}[\cite{LMR9}]
  \label{theo:final}
  Soit $f^0=F(E) \ge 0$ une solution stationnaire continue, non nulle,
  \`a support compact, du syst\`eme \eqref{eq:1}-\eqref{eq:2}, pour
  laquelle il existe $E_0<0$ tel que $F(E)=0$ pour $E \ge E_0$, $F$
  est $C^1$ sur $]-\infty,E_0[$ et $F'<0$ sur $]-\infty,E_0[$.

  Alors $f^0$ est orbitalement stable au sens suivant: pour tous $M>0$
  et $\varepsilon>0$ il existe $\eta>0$ tel que, pour toute donn\'ee
  initiale 
\[
f_{in} \in \cE := \{ g \ge 0, \ g \in L^1\cap L^\infty(\R^6), \ |v|^2
g \in L^1(\R^6) \}
\]
telle que 
\[
\| f_{in} - f^0 \|_{L^1(\R^6)} \le \eta, \quad \cH(f_{in}) \le
\cH(f^0) + \eta, \quad \| f_{in}\|_{L^\infty(\R^6)} \le \| f^0
\|_{L^\infty(\R^6)} + M,
\]
alors toute solution faible issue de cette donn\'ee initiale v\'erifie 
\[
\fa t \ge 0, \quad \drint \left| (1+|v|^2) \big( f(t,x,v) -
  f^0(x-z(t), v) \big) \right| \ddxv \le \varepsilon.
\]
\end{theo}

Avant de d\'etailler la preuve, donnons les id\'ees essentielles: 
\begin{itemize}
\item en s'inspirant d'id\'ees introduites dans la litt\'erature physique
  \cite{Ga,L-B3,WZS,Al} les auteurs d\'emontrent que le hamiltonien
    poss\`ede une propri\'et\'e de monotonie par rapport aux
    r\'earrangements \emph{selon l'\'energie microscopique};
\item apr\`es ce r\'earrangement, on est alors ramen\'e \`a un
  \emph{probl\`eme variationnel de minimisation sur le potentiel
    gravitationnel uniquement}, pour lequel la solution stationnaire
  est bien un minimum isol\'e;
\item enfin pour ce probl\`eme de minimisation r\'eduit, ils
  d\'emontrent une in\'egalit\'e de coercitivit\'e d'Antonov
  g\'en\'eralis\'ee dans ce contexte, et font le lien entre la partie
  radiale de cette in\'egalit\'e et une in\'egalit\'e de
  type Poincar\'e;
\item la fin de la preuve est bas\'ee sur un argument de compacit\'e
  pour des suites minimisantes \og g\'en\'eralis\'ees\fg{} dont le
  r\'earrangement selon l'\'energie microscopique est une suite
  minimisante pour le probl\`eme r\'eduit sur le champ gravitationnel,
  et la compacit\'e est extraite \`a partir de la coercitivit\'e de
  l'\'etape pr\'ec\'edente.
\end{itemize}
Ce travail met donc \`a jour une nouvelle structure variationnelle \og
cach\'ee sous les r\'earrangements selon l'\'energie
microscopique\fg{}, pour laquelle l'approche variationnelle est bien plus
simple et naturelle. Il r\'ev\`ele \'egalement le lien entre la
coercitivit\'e de ce probl\`eme de minimisation r\'eduit et un
probl\`eme d'in\'egalit\'e fonctionnelle de type Poincar\'e.

\subsubsection{R\'earrangement selon l'\'energie microscopique}
\label{sec:rearr-selon-lenerg}

Rappelons tout d'abord la notion classique de \emph{r\'earrangement
  sym\'etrique} (voir par exemple \cite[Chapitre~3]{LL}). \'Etant
donn\'e un ensemble $A \subset \R^6$ mesurable, on d\'efinit son
r\'earrangement sym\'etrique $A^*$ comme \'etant la boule ouverte
centr\'ee en z\'ero et de m\^eme volume que $A$ (pour une norme
donn\'ee sur $\R^6$). \'Etant donn\'ee une fonction $f \ge 0$
int\'egrable sur $\R^6$, on d\'efinit alors son r\'earrangement
sym\'etrique $f^*$ comme \'etant la fonction positive sur $\R^6$ dont
les ensembles de niveau sup\'erieur sont obtenus par r\'earrangement
sym\'etrique des ensembles de niveau sup\'erieur correspondants de
$f$, ce qui donne la formule suivante par int\'egration par tranche
\[
f^*(x,v) = \int_0 ^{+\infty} 1_{\{ f \ge s \}} ^* \dd s \quad \mbox{
  avec } \quad 1_A ^* = 1_{A^*}.
\]
La fonction $f^*$ est alors radialement sym\'etrique, d\'ecroissante,
et \'equimesurable \`a $f$. Rappelons la propri\'et\'e \'el\'ementaire
suivante sur les r\'earrangements sym\'etriques.
\begin{lemm}
  Pour $f \ge 0$ int\'egrable sur $\R^6$ on a 
\[
\drint f^*(x,v) |(x,v)|_{\R^6} \ddxv \le \drint f(x,v) |(x,v)|_{\R^6} \ddxv
\]
o\`u $|(x,v)|_{\R^6}$ d\'esigne la norme consid\'er\'ee sur $\R^6$. 
\end{lemm}

\begin{proof}
  La preuve est tr\`es simple, nous la rappelons pour \'eclairer la
  suite. Pour deux ensembles mesurables $A, B \subset \R^6$ de volume
  fini avec $|A| \le |B|$ (l'autre cas \'etant sym\'etrique), on a 
\[ 
\drint 1_A 1_B \ddxv = | A \cap B | \le | A | = | A^* | = | A^* \cap
B^* | = \drint 1_A ^* 1_B ^* \ddxv.
\]
Or pour $s \ge 0$ donn\'e on a $1_{|(x,v)|_{\R^6} \le s} ^* =
1_{|(x,v)|_{\R^6} \le s}$, et on d\'eduit en utilisant la
pr\'ec\'edente in\'egalit\'e et en int\'egrant par tranche
\[
\drint f 1_{|(x,v)|_{\R^6} \le s} \ddxv \le \drint f^*
1_{|(x,v)|_{\R^6} \le s} \ddxv. 
\]
Puisque $\drint f \ddxv = \drint f^* \ddxv$ on en d\'eduit 
\[
\drint f^* 1_{|(x,v)|_{\R^6} > s} \ddxv \le \drint f
1_{|(x,v)|_{\R^6} > s} \ddxv. 
\]
En int\'egrant finalement selon $s \in [0,+\infty[$, on en d\'eduit le
r\'esultat. 
\end{proof}

On introduit maintenant de mani\`ere similaire le
\emph{r\'earrangement selon l'\'energie microscopique}
\[
E_\phi := \left( \frac{|v|^2}{2} + \phi(x) \right) 
\]
d'un potentiel donn\'e $\phi$ sur $\R^3$ de la mani\`ere
suivante. \'Etant donn\'e un ensemble $A \subset \R^6$ mesurable on
d\'efinit $A^{*\phi}$ son r\'earrangement selon l'\'energie
microscopique $E_\phi$ comme \'etant la \og boule d'\'energie\fg{}
ouverte
\[
A^{*\phi} = \{ (x,v) \ | \ E_\phi(x,v) < E_A \} 
\]
avec $E_A$ choisi tel que $|A^{*\phi}| = |A|$. Il est facile de voir
que
\[
E \to | \{ (x,v) \ | \ E_\phi(x,v) < E \}  |
\]
est une bijection de $[\min E_\phi,0[$ sur $\R_+$. \'Etant donn\'ee
une fonction $f \ge 0$ int\'egrable \` a support compact sur $\R^6$,
on d\'efinit alors $f^{*\phi}$ son r\'earrangement selon l'\'energie
microscopique $E_\phi$ comme \'etant la fonction positive sur $\R^6$
dont les ensembles de niveau sup\'erieur sont obtenus par
r\'earrangement selon l'\'energie microscopique des ensembles de
niveau sup\'erieur correspondants de $f$:
\[
f^{*\phi}(x,v) = \int_0 ^{+\infty} 1_{\{ f \ge s \}} ^{*\phi} \dd s \quad \mbox{
  avec } \quad 1_A ^{*\phi} = 1_{A^{*\phi}}.
\]
La fonction $f^{*\phi}$ est alors une fonction de $E_\phi$, \`a support
compact, d\'ecroissante en l'\'energie microscopique $E_\phi$, et
\'equimesurable \`a $f$. On a la propri\'et\'e suivante qui rappelle
le lemme \'el\'ementaire ci-dessus, et dont nous nous servirons par la
suite:
\begin{lemm}
  Pour $f \ge 0$ int\'egrable \`a support compact sur $\R^6$ on a
\[
\drint f^{*\phi}(x,v) E_\phi(x,v) \ddxv \le \drint f(x,v) E_\phi(x,v) \ddxv.
\]
\end{lemm}

\begin{proof}
En raisonnant comme pr\'ec\'edemment on a 
\[ 
\drint 1_A 1_B \ddxv \le \drint 1_A ^{*\phi} 1_B ^{*\phi} \ddxv.
\]
D'o\`u, pour $s \in [-\min E_\phi,0[$, puisque $1_{E_\phi(x,v) \le
  s} ^{*\phi} = 1_{E_\phi(x,v) \le s}$, en int\'egrant par tranche
\[
\drint f 1_{E_\phi(x,v)  \le s} \ddxv \le \drint f^{*\phi}
1_{E_\phi(x,v) \le s} \ddxv
\]
et puisque $\drint f \ddxv = \drint f^{*\phi} \ddxv$ on en d\'eduit
\[
\drint f^{*\phi} 1_{E_\phi(x,v) > s} \ddxv \le \drint f
1_{E_\phi(x,v)>s} \ddxv. 
\]
En int\'egrant finalement selon $s \in [-\min E_\phi,0[$, on obtient
\[
\drint f^{*\phi} \left( E_\phi(x,v) - \min E_\phi\right) \ddxv \le
\drint f 
\left( E_\phi(x,v) - \min E_\phi \right) \ddxv
\]
d'o\`u le r\'esultat.
\end{proof}

On va par la suite utiliser le r\'earrangement de $f$ \emph{selon
  l'\'energie microscopique cr\'e\'ee par la fonction $f$
  elle-m\^eme}, que nous noterons $\hat f = f^{*\phi_f}$ avec, comme
pr\'ec\'edemment, $\phi_f = - (1/(4 \pi |x|)) \ast \rho_f$. On voit que
l'on obtient ainsi une op\'eration de r\'earrangement $f \to \hat f$
\emph{tr\`es fortement non-lin\'eaire}. Remarquons imm\'ediatement que
la solution stationnaire sph\'erique est un point fixe de ce
r\'earrangement non-lin\'eaire: $\hat f^0 = (f^0)^{*\phi_{f^0}} =
f^0$.

\subsubsection{Monotonie du hamiltonien et hamiltonien r\'eduit}
\label{sec:monot-du-hamilt}

On d\'efinit la fonctionnelle 
\[
\cJ_{f^*}(\phi) := \cH(f^{*\phi}) + \frac12 \| \Dx \phi - \Dx \phi_{f*\phi}
\|_{L^2(\R^3)} ^2
\]
et on montre la propri\'et\'e de monotonie suivante. 
\begin{prop}
  \label{prop:mon-ham}
  Si l'on consid\`ere $f \in \cE$ et $\hat f = f^{*\phi_f}$, alors
  \[
  \cH(f) \ge \cJ_{f^*} (\phi_f) \ge \cH(\hat f)
\]
avec \'egalit\'e si et seulement si $f = \hat f$. 
\end{prop}

\begin{proof}[\'Ebauche de preuve]
  La preuve repose sur le lemme pr\'ec\'edent. Par un calcul que nous
  avons d\'ej\`a fait 
\[
\cH(f) = \cH(g) + \frac12 \| \Dx \phi_f - \Dx \phi_g \|_{L^2(\R^3)} ^2
+ \drint \left( \frac{|v|^2}{2} + \phi_f(x) \right) (f-g) \ddxv 
\]
pour deux fonctions $f,g \in \cE$, d'o\`u avec $g = \hat f$: 
\[
\cH(f) = \cJ_{f^*}(\phi_f) + \drint \left( \frac{|v|^2}{2} + \phi_f(x)
\right) (f-\hat f) \ddxv
\]
et l'on conclut gr\^ace au lemme pr\'ec\'edent appliqu\'e \`a 
\[
\drint \left( \frac{|v|^2}{2} + \phi_f(x)
\right) (f-\hat f) \ddxv \ge 0.
\]
Le cas d'\'egalit\'e se traite en \'etudiant le cas d'\'egalit\'e dans
l'in\'egalit\'e de monotonie du r\'earrangement selon l'\'energie
microscopique. 
\end{proof}

Il s'av\`ere que la diff\'erence $\cJ_{f^*}- \cJ_{f^0}$ est facilement
contr\^olable par des normes sans d\'eriv\'ee:
\[
\cJ_{f^*} (\phi) - \cJ_{f^0}(\phi) \ge - \| \phi_f \|_{L^\infty(\R^3)}
\| f^* - (f^0)^* \|_{L^1(\R^6)},
\]
et l'on ainsi peut se contenter d'\'etudier la coercitivit\'e de la
fonctionnelle $\cJ_{f^0}$. Nous noterons $\cJ(\phi) :=
\cJ_{f^0}(\phi)$ et nous appellerons cette fonctionnelle
\emph{hamiltonien r\'eduit}; elle n'agit que sur le potentiel $\phi$.

\subsubsection{In\'egalit\'e de coercitivit\'e d'Antonov
  g\'en\'eralis\'ee pour le hamiltonien r\'eduit}
\label{sec:strict-convexite-du}

On va maintenant \'etudier les propri\'et\'es de convexit\'e de $\cJ$
au voisinage de $\phi_{f^0}$, et ainsi montrer que $\phi_{f^0}$ est un
\emph{minimum local} de $\cJ$. On d\'efinit un espace de potentiels
admissibles
\[
\cX := \left\{ \phi \in C^0(\R^3) \quad | \quad \phi \le 0, \ \lim_{\infty}
  \phi = 0, \ \Dx \phi \in L^2(\R^3), \ \inf_{x \in \R^3}
  (1+|x|)|\phi(x)| >0 \right\}.
\]
On peut alors montrer la proposition suivante.
\begin{prop}
  \label{prop:coercJ}
  Il existe des constantes $c_0$, $\delta_0 >0$ et une application
  continue $\phi \to z_\phi$ de $\dot H^1(\R^3)$ (l'espace de Sobolev
  homog\`ene) dans $\R^3$ telles que pour $\phi \in \cX$ tel que
\[
\inf_{z \in \R^3} \Big[ \| \phi - \phi_{f^0}(\cdot -z)
\|_{L^\infty(\R^3)} + \| \Dx \phi - \Dx \phi_{f^0}(\cdot - z)
\|_{L^2(\R^3)} \Big] < \delta_0
\]
alors 
\[
\cJ(\phi) - \cJ(\phi_{f^0}) \ge c_0 \| \Dx \phi - \Dx \phi_{f^0}(\cdot - z_\phi)
\|_{L^2(\R^3)}.
\]
\end{prop}

\begin{proof}[\'Ebauche de preuve]
On d\'ecompose la d\'emarche en plusieurs \'etapes.

\sk
\noindent
\'Etape 1: \emph{D\'eveloppement de Taylor.} La premi\`ere phase
calculatoire est d'\'ecrire le d\'eveloppement de Taylor \`a l'ordre
$2$ de $\cJ$ avec un reste contr\^ol\'e explicitement: 
\[
\cJ(\phi) - \cJ(\phi_{f^0}) = \frac12 D^2 \cJ(\phi_{f^0})
(\phi-\phi_{f^0},\phi-\phi_{f^0}) + o\left(\| \phi -
  \phi_{f^0}\|_{L^\infty(\R^3)}\right) \| \Dx \phi - \Dx \phi_{f^0} 
\|_{L^2(\R^3)} ^2
\]
avec le terme d'ordre $1$ qui s'annule et
\[
D^2 \cJ(\phi_{f^0})(h,h) := \rint | \Dx h |^2 \dd x - \drint |F'(E)|
\big[ h(x) - (\cP h) (x,v) \big]^2 \ddxv 
\]
o\`u l'on rappelle que, sans autre pr\'ecision, $E=E_{\phi_{f^0}}=
(|v|^2/2+\phi_{f^0}(x))$, et la projection~$\cP$ est d\'efinie par 
\[
(\cP h)(x,v) := \frac{\rint \left( \frac{|v|^2}{2} + \phi_{f^0}(x) -
    \phi_{f^0}(y) \right)^{1/2} _+ h(y) \dd y}{\rint \left( \frac{|v|^2}{2}
    + \phi_{f^0}(x) - \phi_{f^0}(y) \right)^{1/2} _+ \dd y}.
\]
C'est un op\'erateur de projection sur les fonctions de $E$
uniquement. Dans le cas radial on retrouve ainsi l'op\'erateur de
projection sur le noyau de l'op\'erateur $v \cdot \nabla_x - \nabla_x
\phi_{f^0} \cdot \nabla_v$ que nous avons d\'ej\`a rencontr\'e dans le
travail \cite{GL}. 

\sk 
\noindent 
\'Etape 2: \emph{In\'egalit\'e d'Antonov g\'en\'eralis\'ee}. 
Si l'on d\'efinit 
\[
\cL h = -\Delta_x h - \rint |F'(E)| (1-\cP) \dd v
\] 
on obtient
\[
\langle \cL h, h \rangle_{L^2(\R^3)} = D^2 \cJ(\phi_{f^0})(h,h)
\]
et l'on ram\`ene le probl\`eme \`a l'\'etude de l'op\'erateur
$\cL$; on a alors la proposition suivante. 
\begin{prop}
  L'op\'erateur $\cL$ est positif, c'est une perturbation compacte du
  laplacien sur $\dot H^1(\R^3)$, son noyau est donn\'e par
\[
\mbox{{\em Ker}}(\cL) = 
\mbox{{\em Vect}}\left\{ \partial_{x_1}
  \phi_{f^0}, \partial_{x_2} \phi_{f^0}, \partial_{x_3} \phi_{f^0} \right\}
\]
et on a donc 
\[
\fa h \in \dot H^1(\R^3), \quad \langle \cL h, h \rangle_{L^2(\R^3)}
\ge c_0 \| \Dx h \|_{L^2(\R^3)} ^2 - \frac{1}{c_0} \sum_{i=1} ^3
\left( \rint h \Delta_x (\partial_{x_i} \phi_{f^0}) \dd x \right)^2
\]
pour une certaine constante $c_0>0$. 
\end{prop}

On d\'ecompose $h$ en partie radiale et compl\'ementaire orthogonal 
\[
h = h_0 + h_1, \quad h_0 \in \dot H^1_{rad}(\R^3), \ h_1 \in
\left(\dot H^1_{rad}(\R^3)\right)^\bot.
\]
La positivit\'e selon la composante $h_1$ est plus simple \`a traiter
car $\Pi h_1=0$ et son \'etude se ram\`ene donc \`a celle de
l'op\'erateur de Schr\"odinger $\cA$ que nous avons d\'ej\`a
\'etudi\'e plus haut dans la preuve de la
proposition~\ref{prop:weinsteinlin}. Le noyau $\mbox{Ker}(\cL)$ de
l'\'enonc\'e s'en d\'eduit en particulier.

Sur la composante radiale $h_0$ on a l'in\'egalit\'e
\[
\fa h \in \dot H^1_{rad}(\R^3), \ h \not =0, \quad 
\langle \cL h, h \rangle_{L^2(\R^3)} >0. 
\]
Cette propri\'et\'e peut se d\'emontrer comme dans la preuve de
Guo et Lin~\cite{GL} que nous avons discut\'ee pr\'ec\'edemment \`a la
sous-section~\ref{sec:appr-non-vari}, en montrant qu'elle se r\'eduit
\`a l'in\'egalit\'e de coercitivit\'e d'Antonov d\'emontr\'ee \`a la
proposition~\ref{prop:Antonov-simple}.

Lemou, M\'ehats et Rapha\"el proposent une autre preuve int\'eressante
de cette propri\'et\'e et de la proposition~\ref{prop:Antonov-simple},
en faisant le parall\`ele avec la d\'emonstration d'une
\emph{in\'egalit\'e de type Poincar\'e}. Ils
adaptent la strat\'egie de preuve de H\"ormander \cite{Ho-i-1,Ho-i-2},
et utilisent une in\'egalit\'e fonctionnelle de type Hardy.

Donnons l'id\'ee g\'en\'erale de cet argument. On introduit
l'op\'erateur suivant sur les fonctions radiales, exprim\'e dans les
variables $E$ et $r$:
\[
T f (E,r) = \frac{1}{r^2 \sqrt{ 2 (E- \phi_{f^0}(r))}} \partial_r f =
\frac{1}{r^2 |v|} \partial_r f, 
\]
et on v\'erifie que $\cP h =0$ implique $h = T \tilde h$ pour un certain
$\tilde h$. On calcule alors par int\'egration par parties et
in\'egalit\'e de Cauchy-Schwarz (un argument d'approximation
suppl\'ementaire est n\'ecessaire, que nous n'\'evoquons pas ici)
\[
\drint |F'(E)| (h-\cP h)^2 \ddxv \le \| \Dx h \|_{L^2(\R^3)} \left( 3
  \drint \rho_{f^0}(r) \frac{\tilde h ^2}{4 r^4 (E-\phi_{f^0}(r))^2} |F'(E)| \ddxv
\right)^{1/2}.
\]
On montre alors l'in\'egalit\'e de type Hardy suivante 
\[
\left( 3 \drint \left( \rho_{f^0}(r) + \frac{(\phi_{f^0})'(r)}{r} \right)
  \frac{\tilde h ^2}{4 r^4(E-\phi_{f^0}(r))^2} |F'(E)| \ddxv \right) \le
\drint |F'(E)| |T \tilde h|^2 \ddxv,
\]
ce qui, combin\'e avec l'in\'egalit\'e pr\'ec\'edente, donne
\[
\| \Dx h \|_{L^2(\R^3)} ^2 - \drint |F'(E)| (h-\cP h)^2 \ddxv \ge 3
\drint \frac{(\phi_{f^0})'(r)}{r} \frac{\tilde h ^2}{4 r^4(E-\phi_{f^0}(r))^2}
|F'(E)| \ddxv 
\]
et conclut la preuve de positivit\'e. L'in\'egalit\'e de Hardy se
d\'emontre en remarquant que 
\[
(T \tilde h)^2 = T \left( \tilde h_1 ^2 \tilde h_2 T \tilde h_2 \right) -
\frac{T^2 \tilde h_2}{\tilde h_2} \tilde h ^2 \quad \mbox{ avec }
\quad \tilde h = \tilde h_1 \tilde h_2,
\]
puis 
\[
 - \frac{T^2 \tilde h_2}{\tilde h_2} = \frac{3}{4 r^4
   (E-\phi_{f^0}(r))^2} \left( \rho_{f^0}(r) + \frac{\phi_{f^0}(r)}{r}
 \right) \quad \mbox{ pour } \quad \tilde h_2 = r^3
 (2(E-\phi_{f^0}(r)))^{3/2}.
\]

\sk 
\noindent 
\'Etape 3: \emph{Traitement du noyau par modulation}.  On ajuste
finalement la fonction de translation $z_\phi$ au moyen d'un
th\'eor\`eme des fonctions implicites pour annuler les d\'efauts de
coercitivit\'e, {\it i.e.}, les termes n\'egatifs dans la
proposition ci-dessus.
\end{proof}

\subsubsection{Compacit\'e des suites minimisantes et r\'esolution de
  la conjecture}
\label{sec:compacite-des-suites}

On peut montrer la compacit\'e de certaines suites minimisantes
g\'en\'eralis\'ees $f_n$ au sens suivant.
\begin{prop}\label{prop:suites}
Si $f_n \in \cE$ v\'erifie 
\[
\sup_{n \ge 0} \inf_{z \in \R^3} \Big[ \| \phi_{f^n} - \phi_{f^0}(\cdot -z)
\|_{L^\infty(\R^3)} + \| \Dx \phi_{f^n} - \Dx \phi_{f^0}(\cdot - z)
\|_{L^2(\R^3)} \Big] < \delta_0
\]
et
\[
\lim_{n \to \infty} \| f_n ^* - (f^0)^* \|_{L^1(\R^6)} = 0, \quad
\liminf_{n \to \infty} \cH(f_n) \le \cH(f^0),
\]
alors 
\[
\lim_{n \to \infty} \drint \left| (1+|v|^2) \big( f_n(x,v) -
  f^0(x-z_{\phi_{f^n}}, v) \big) \right| \ddxv =0.
\]
\end{prop}

\begin{proof}[\'Ebauche de preuve] 
  La preuve est faite en deux \'etapes. Tout d'abord le contr\^ole de
  coercitivit\'e pr\'ec\'edent implique sans difficult\'es que
\[
\lim_{n \to \infty} \| \Dx \phi_{f^n} - \Dx \phi_{f^0}(\cdot - z_{\phi_{f^n}})
\|_{L^2(\R^3)} =0.
\]
Ensuite on note $\bar f_n(x,v) = f_n (x + z_{\phi_{f^n}},v)$ et l'on
revient au hamiltonien complet en utilisant l'identit\'e 
\[
\cH(\bar f_n) - \cH(f^0) + \frac12 \| \Dx \phi_{\bar f^n} - \Dx \phi_{f^0}
\|_{L^2(\R^3)} = \drint E_{\phi_{f^0}} (\bar f_n - f^0) \ddxv.
\]
\`A partir des hypoth\`eses et de la convergence d\'ej\`a
d\'emontr\'ee, on a 
\[
\left\{ 
\begin{array}{l} \ds 
\limsup_{n \to \infty} \left( \cH(\bar f_n) - \cH(f^0) \right) = \limsup_{n \to
  \infty} \left( \cH(f_n) - \cH(f^0)  \right) \le 0,\vs \\ \ds 
\| \Dx \phi_{\bar f^n} - \Dx \phi_{f^0} \|_{L^2(\R^3)} = \| \Dx \phi_{f^n}
- \Dx \phi_{f^0}(\cdot - z_{\phi_{f^n}}) \|_{L^2(\R^3)} \to 0,
\end{array}
\right.
\]
et on en d\'eduit 
\[
\limsup_{n \to \infty} \drint E_{\phi_{f^0}} (\bar f_n - f^0) \ddxv \le
0. 
\]
L'hypoth\`ese $(f_n)^* \to (f^0)^*$ implique par ailleurs  
\[
\lim_{n \to \infty} \drint E_{\phi_{f^0}} (f^0 - \bar f_n ^{* \phi_{f^0}}) \ddxv = 0
\]
d'o\`u l'on d\'eduit 
\[
\limsup_{n \to \infty} \drint E_{\phi_{f^0}} (\bar f_n - \bar f_n ^{* \phi_{f^0}}) \ddxv \le
0.
\]
De par la monotonie du r\'earrangement $\drint E_{\phi_{f^0}} (\bar f_n
- \bar f_n ^{* \phi_{f^0}}) \ddxv \ge 0$, on en d\'eduit finalement
\[
\lim_{n \to \infty} \drint E_{\phi_{f^0}} (\bar f_n -\bar  f_n ^{* \phi_{f^0}})
\ddxv =0.
\]
Il suffit ensuite de montrer que la saturation de cette in\'egalit\'e
de r\'earrangement, combin\'ee \`a l'hypoth\`ese $(f_n)^* \to (f^0)^*$
implique que
\[
\lim_{n \to \infty} \| f_n - f^0 \|_{L^1(\R^6)} = 0.
\]
En combinant cette convergence avec l'hypoth\`ese de limite
sup\'erieure sur le hamiltonien ainsi que la convergence du potentiel,
on obtient finalement la convergence de l'\'energie cin\'etique 
\[
\lim_{n \to \infty} \drint |v|^2 f_n \ddxv = \drint |v|^2 f^0 \ddxv
\]
ce qui conclut la preuve de la proposition~\ref{prop:suites}. 
\end{proof}

La fin de la preuve du th\'eor\`eme \ref{theo:final} se fait ensuite
en combinant la proposition~\ref{prop:suites},
l'in\'egalit\'e d'interpolation de la
proposition~\ref{prop:interpol-clef}, ainsi que la contractivit\'e du
r\'earrangement sym\'etrique $\| f^* - (f^0)^* \|_{L^1(\R^6)} \le \| f
- f^0 \|_{L^1(\R^6)}$.

\section{\textup{Conclusion et probl\`emes ouverts}}
\label{sec:concl-et-probl}

Ce probl\`eme de stabilit\'e des galaxies est un exemple int\'eressant
de recherche math\'ematique nourrie par une question concr\`ete
pos\'ee par la physique th\'eorique.  Nous essayons pour terminer de
soulever quelques questions ouvertes d'ordre math\'ematique, en
sugg\'erant des liens avec d'autres travaux.

Tout d'abord, la premi\`ere question naturelle du point de vue de la
pertinence physique des r\'esultats est de \emph{quantifier} la taille
du voisinage de stabilit\'e orbitale. Cela para\^it maintenant une
t\^ache plus abordable avec la nouvelle th\'eorie de Lemou, M\'ehats
et Rapha\"el; il s'agit essentiellement de rendre explicites, ou tout
au moins constructives, les constantes de coercitivit\'e dans les
in\'egalit\'es fonctionnelles utilis\'ees.

L'autre question naturelle est de sortir du cadre strictement monotone
pour la solution stationnaire $f^0(E)$. Par exemple, nous pouvons
d\'ej\`a nous demander si, au niveau des solutions stationnaires,
localement au voisinage d'une solution orbitalement stable, il est
possible de d\'emontrer un th\'eor\`eme de param\'etrisation bijective
des solutions stationnaires par les conservations du syst\`eme, dans
le m\^eme esprit que le travail r\'ecent de Choffrut et Sver\'ak
\cite{CS} sur l'\'equation d'Euler incompressible en dimension $2$.

Cependant, nous pourrions nous attendre plus g\'en\'eralement, au
niveau dynamique, \`a la stabilit\'e orbitale autour d'une solution
stationnaire \og presque\fg{} monotone, et donc proche des solutions
orbitalement stables que nous avons \'etudi\'ees. Une premi\`ere
t\^ache serait ici de clarifier au niveau math\'ematique les
instabilit\'es cr\'e\'ees par des perturbations non radiales de
mod\`eles sph\'eriques anisotropes.

Les m\'ethodes variationnelles semblent n\'eanmoins trouver leur
limite, et cela soul\`eve la question de revenir \`a nouveau \`a une
approche directe par lin\'earisation. Cela nous am\`ene \'egalement
\`a faire une autre remarque importante sur les travaux que nous avons
pr\'esent\'es: \emph{ceux-ci n'utilisent pas la dynamique proprement
  dite de l'\'equation, mais uniquement ses invariants}. M\^eme si les
r\'esultats obtenus sont dynamiques, le c\oe ur conceptuel de ces
m\'ethodes n'utilise pas la dynamique. C'est une force de ces
approches, qui leur conf\`ere une grande robustesse pour traiter des
mod\`eles g\'en\'eraux et manipuler des solutions tr\`es faibles, mais
c'est \'egalement une faiblesse d\`es lors que l'on sort d'un cadre
parfaitement variationnel.

Par cons\'equent, il serait int\'eressant d'explorer les dialogues
possibles avec les r\'esultats de stabilit\'e non-lin\'eaire obtenus
dans \cite{MV}. Ces derniers r\'esultats utilisent des espaces
fonctionnels tr\`es r\'eguliers, c'est donc en ce sens l'extr\^eme
oppos\'e des m\'ethodes de stabilit\'e orbitale que nous avons
pr\'esent\'ees. En particulier, cela implique qu'il faut \og
traquer\fg{} les oscillations du syst\`eme dans les estimations de
r\'egularit\'e, alors que les espaces de Lebesgue utilis\'es dans les
th\'eories de stabilit\'e orbitale ne \og voient\fg{} pas les
oscillations de la variable de vitesse produites par le m\'elange de
phase. Les hypoth\`eses sur les donn\'ees initiales sont donc bien
plus fortes, mais cela permet \'egalement d'obtenir une information
plus pr\'ecise sur le comportement asymptotique du syst\`eme, et de
montrer la stabilit\'e non-lin\'eaire autour de solutions
stationnaires lin\'eairement stables, mais qui ne v\'erifient pas un
probl\`eme variationnel.

Les r\'esultats de \cite{MV} s'appliquent \`a l'\'equation de
Vlasov-Poisson gravitationnelle, mais uniquement dans le cas non
physique d'un domaine p\'eriodique en espace. Le cas d'un syst\`eme
auto-gravitant qui \og cr\'ee sa propre g\'eom\'etrie\fg{} et son
propre confinement au cours du temps repr\'esente un d\'efi
probablement difficile mais aussi tr\`es int\'eressant pour cette
approche par lin\'earisation: \`a l'inverse de \cite{MV} o\`u \og
l'amortissement Landau\fg{} produit une convergence vers z\'ero du
champ moyen asymptotiquement, on ne conna\^\i t plus ici \`a l'avance
quelle doit \^etre la limite du champ moyen lorsque $t \to
+\infty$. Un d\'efi conceptuel similaire se pose pour l'\'equation
d'Euler incompressible en dimension~$2$ ainsi que pour l'\'equation de
Vlasov-Poisson avec un confinement magn\'etique.

Pour terminer, nous mentionnerons le probl\`eme int\'eressant, mais
probablement pour le moment hors d'atteinte tant que les questions
pr\'ec\'edentes et la limite de champ moyen ne sont pas mieux comprises,
de faire le lien entre les r\'esultats de stabilit\'e obtenus pour
l'\'equation de Vlasov-Poisson gravitationnelle, et les r\'esultats de
stabilit\'e asymptotique (on pense par exemple \`a la th\'eorie KAM)
pour le probl\`eme \`a $N$ corps, dans la limite $N \to \infty$.

\end{document}